\documentclass[final]{IEEEtran}
\IEEEoverridecommandlockouts

\usepackage[margin=0.75in]{geometry}
\usepackage{amsmath,amssymb, ntheorem}
\usepackage{psfrag}
\usepackage{cite}
\usepackage{graphicx}
\usepackage{subcaption}
\usepackage[font=small,labelfont=bf,belowskip=0pt]{caption}
\usepackage{graphics}
\usepackage{verbatim}
\usepackage{epstopdf}
\usepackage{float}
\usepackage{framed}
\usepackage{textcomp}
\usepackage{bbm}
\usepackage{algorithm}
\usepackage{algorithmicx}
\usepackage{algpseudocode}
\usepackage{accents}
\usepackage{enumerate}
\usepackage{stmaryrd}
\usepackage{verbatim}

\restylefloat{figure}
\usepackage{color,xspace}
\usepackage{tikz}
\usepackage{upgreek}
\usepackage{mathrsfs}
\usetikzlibrary{automata,positioning,shapes,fit,patterns}

\usepackage[normalem]{ulem}
\makeatletter
\renewtheoremstyle{plain}
{\item{\theorem@headerfont ##1\ ##2\theorem@separator}~}
{\item{\theorem@headerfont ##1\ ##2\ (##3)\theorem@separator}~}
\makeatother

{\theoremheaderfont{\upshape\bfseries}
\theorembodyfont{\normalfont\slshape}
	
\newtheorem{defn}{Definition}
\newtheorem{lem}{Lemma}
\newtheorem{cor}{Corollary}
\newtheorem{theorem}{Theorem}

\newtheorem{prop}{Proposition}
}

\floatstyle{plain}
\restylefloat{figure}
\usepackage{pstool}
\usepackage{url}

\newcommand{\real}{\mathbb{R}}

\newcommand{\lproj}{\llbracket}
\newcommand{\rproj}{\rrbracket}

\newcommand{\Znonnegative}{\mathbbm{Z}_{\geq 0}}

\newcommand{\Nat}{\mathbbm{N}}

\newcommand{\goesto}{\rightarrow}
\newcommand{\curlyleq}{\preccurlyeq}

\newcommand{\curlygeq}{\succcurlyeq}

\newcommand{\Hmax}{H_{\text{max}}}
\newcommand{\Bcap}{B_{\text{cap}}}
\newcommand{\Bcrit}{B_{\text{crit}}}

\renewcommand{\Pr}{\mathbb{P}}

\newcommand{\Amode}[1]{A_{\sigma(#1)}}
\newcommand{\mode}[1]{\sigma(#1)}
\newcommand{\Recurrent}{\mathcal{R}}
\newcommand{\Trans}[2]{\Phi(#1|#2)}
\newcommand{\pmin}{p}

\newcommand{\Ltrans}{\mathscr{L}}
\newcommand{\Gtrans}{\mathscr{G}}
\newcommand{\Pd}{\rho}
\newcommand{\Intensity}{\mathcal{I}}
\newcommand{\Dampening}{\mathcal{D}}
\newcommand{\Lyap}{V}

\newcommand{\Interval}{\mathcal{I}}
\newcommand{\HarvestSpace}{\mathcal{H}}

\newcommand{\NP}{\mathcal{NP}}

\newcommand{\E}{\mathbb{E}}
\newcommand{\Tr}{\operatorname{Tr}}

\newcommand{\MJLSmat}[1]{\psi_{(s^\prime | s)}^{\{#1\}}}

\newcommand{\naive}{na{\"{i}}ve }
\newcommand{\Energy}{\mathcal{E}}
\newcommand{\energy}{\varepsilon}
\newcommand{\energythresh}{\bar{\energy}}

\newcommand{\MembershipEvent}[1]{\mathcal{T}_{#1}}
\newcommand{\Target}{T}

\newcommand{\Mode}{\theta}
\newcommand{\ModeTime}{\tau_{\Mode(j)}}
\newcommand{\ModeChangeCount}{M}

\newcommand{\sympossemidef}{\mathbb{S}_+}

\newcommand{\disturbance}{\omega}
\newcommand{\noise}{\omega}

\newcommand{\oprocendsymbol}{\hbox{$\bullet$}}
\newcommand{\oprocend}{\relax\ifmmode\else\unskip\hfill\fi\oprocendsymbol}

\newenvironment{proof}{\noindent \textbf{Proof :}}{\oprocend}

\newtheorem{remark}{\textbf{Remark}}{}
{}

\pgfdeclarelayer{background}
\pgfdeclarelayer{foreground}
\pgfsetlayers{background,main,foreground}
\begin{document}
	
	\captionsetup{belowskip=0pt} 
	\title{Stability of Control Systems with \\ Feedback from Energy Harvesting Sensors}
	\author{Nicholas J. Watkins, Konstantinos Gatsis, Cameron Nowzari, and George J. Pappas\thanks{N.J. Watkins, K. Gatsis, and G.J. Pappas are with the Department of Electrical and Systems Engineering, University of Pennsylvania, Pennsylvania, PA 19104, USA, {\tt\small \{nwatk,kgatsis,pappasg\}@upenn.edu}; C. Nowzari is with the Department of Electrical and Computer Engineering, George Mason University, Fairfax, VA 22030, USA, {\ttfamily\small cnowzari}@gmu.edu.} }
	\maketitle
	\begin{abstract}
		In this paper, we study the problem of certifying the stability of a closed loop system which receives feedback from an energy harvesting sensor.  This is important, as energy harvesting sensors are recharged stochastically, and may only be able to provide feedback intermittently.  Thus, stabilizing plants with feedback provided by energy harvesting sensors is challenging in that the feedback signal is only available stochastically, complicating the analysis of the closed-loop system.  As the main contribution of the paper, we show that for a broad class of energy harvesting processes and transmission policies, the plant state process can be modeled as a Markov jump linear system (MJLS), which thereby enables a rigorous stability analysis.  We discuss the types of transmission policies and energy harvesting processes which can be accommodated in detail, demonstrating the generality of the results.
	\end{abstract}
	\section{Introduction} \label{sec:intro}
	
	Energy harvesting technology - which allows for a device's battery to be recharged online by interacting with the environment - will play a significant role in the development of future smart technologies.  Indeed, principled use of energy harvesting technologies will allow for the safe use of sensors in remote locations without the need for explicit, periodic maintenance or replacement.  In recent years, much progress has been made in understanding how to use energy harvesting devices in networking and communications applications \cite{Ulukus2015, Sudevalayam2011, Basagni2013}.  However, there is relatively little literature detailing how energy harvesting sensors can be used in control applications, where the closed-loop system's dynamical behavior is of significant importance.   A key desirable property of many control systems is provable closed-loop stability, which often serves as formal means for guaranteeing safe system operation. 
	
	Ensuring the stability of a plant which receives its feedback signal from an energy harvesting sensor is a challenging problem.  Since the sensor's energy is restored by interactions with the environment (e.g. by leveraging vibrations in a mechanical process \cite{Beeby2007a}, differences in temperature between a surface and the environment \cite{Tan2011}, or the presence of ambient solar light \cite{Raghunathan2005,Alippi2008}), it will typically be the case that feedback can only be provided intermittently.  Indeed, in this context, a sensor may only provide a feedback signal with positive probability when it has sufficient energy to transmit the signal.  Since the process by which the sensor's battery is restored (i.e., the \emph{energy harvesting process}) will often have a significant stochastic component, analysis of the plant's state is difficult.  In particular, correlations between the energy harvesting process and the plant state evolution makes the closed-loop dynamics of the system complicated.
	
	Currently, most works which have considered the interface of dynamical systems and energy harvesting sensors have either not explicitly addressed closed-loop stability of the process, or have done so under conservative assumptions.  Indeed, the earliest known work on sensing of dynamical systems with energy harvesting sensors \cite{Nayyar2013} explicitly considers the problem of minimizing the expected value of the state estimation error - it does not explicitly address system stability.  Similarly \cite{Nourian2014b} and \cite{Ozel2016} find conservative conditions under which the estimation error of Kalman filters running on energy harvesting sensors will remain bounded. Works considering the closed loop stability of the system are limited.  
	
	In particular, \cite{Knorn2017, Li2016, Calvo2017} have studied controllers which guarantee closed-loop stability under conservative assumptions. While these do not directly assume that the open-loop system is stable, they indirectly assume that at every time increment in the process, enough energy arrives so that the sensor can communicate reliably enough with the plant to guarantee uniform decay in a norm of the system's plant state.  Since this assumption directly implies that a positive amount of energy is harvested by the sensor at all times with a positive probability, these results are restrictive.  Indeed, it seems that in many practical settings, the energy harvesting process will \emph{not} supply energy to the sensor for long periods of time.  This is indeed the case for solar cells, and also when the sensor is deterministically recharged according to a fixed schedule.  As such, it is clear that a more general stability analysis is needed; we perform such an analysis in this paper.  In particular, we consider the problem of certifying the closed-loop stability of a plant when supplied with feedback in accordance to a fixed, memoryless, energy-causal transmission policy, where the sensor is recharged by process modeled by a function of a Markov process.
	
	\emph{The primary contribution of this paper} is an efficiently computable stability certification method for systems which receive feedback information from an energy harvesting sensor following a known transmission strategy, restored with energy from a known stochastic energy harvesting process.  To accomplish this, we show that for a large class of transmission policies and energy harvesting processes, such systems can be modeled as a Markov jump linear system (MJLS) with a mode transition process which is defined on a state space whose size grows mildly with the size of the energy harvesting processes' transition matrix.  We then adapt  stability results from the MJLS literature to our setting.
	
	In order to demonstrate the generality of the proposed stability certification method, we discuss in detail the types of transmission policies and energy harvesting processes which can be accommodated.  We show that any memoryless transmission policy can be accommodated into our framework.  This is important, as this is a sufficiently broad classification so as to be useful for many systems.  Indeed, we demonstrate that memoryless policies are all which are required in order to stabilize the system when the plant is scalar, and intelligently designed memoryless policies often suffice to stabilize nonscalar plants.  Likewise, we show that any energy harvesting model which can be posed as a function of a finite-state Markov chain can be accommodated into our framework.  This is important, as many important types of energy sources can be modeled as such, as we demonstrate in Section \ref{sec:examples}.  The work presented here differs from our preliminary conference paper \cite{Watkins2017b} in that it extends the technical results from undisturbed scalar plants to arbitrary linear plants subject to stochastic disturbances, and provides an extended discussion on modeling different types of energy harvesting sources within the considered framework.
	
	\paragraph*{\textbf{Organization}} 
	The paper is organized as follows.  The architecture of the system we study is presented in Section \ref{sec:problem}, along with  
	a formal problem statement.  The main results of our paper are contained in Section \ref{sec:stab_cert}, in which we propose a test for certifying the stability of an energy harvesting system under a fixed transmission policy.  Section \ref{sec:principles_policy} provides principles regarding the design of transmission policies for energy harvesting control systems.  Section \ref{sec:examples} contains examples of energy harvesting processes, with each serving to demonstrate the proposed method's applicability to a different potential application.  Section \ref{sec:conclusion} concludes the paper. \oprocend
	
	\paragraph*{\textbf{Notation}}
	We denote by $\Znonnegative$ the set of non-negative integers, and for each $k \in \Znonnegative,$ we denote by $[k]_0$ the set of non-negative integers $\{0,1, ..., k\}.$  We denote by $\lproj k \rproj_{a}^b$ the projection of $k$ into the interval $[a,b].$  Let $\Pr$ be a probability measure, $Q$ an event which is measurable with respect to $\Pr,$ and $X$ a random variable.  We use the notation $\Pr_{X}(Q)$ for the conditional probability of $Q$ given $X$ when writing the explicit expression $\Pr(Q | X)$ is too cumbersome.  \oprocend
	
	\section{Problem Statement}
	\label{sec:problem}
	
	A visual representation of the system architecture we study is shown in Figure \ref{fig:arch}.  This models a setting in which an energy harvesting sensor communicates over a stochastic communication channel to stabilize the evolution of a plant.  The sensor stores energy in its battery, and restores its charge via a stochastic energy harvesting process.  The control designer's role in this system's evolution is in designing transmission policies which determine when and how energy should be used in order to affect the evolution of the plant's state vector.  The principle question we address in this text is that of determining if a chosen transmission policy stabilizes the closed-loop evolution of the plant's state.  We now detail mathematical models for each component of the system, and provide a formal problem statement.
	\begin{figure}[!t]
		\includegraphics[width=0.48\textwidth]{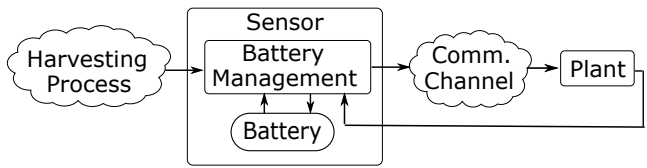}
		\caption{\small Energy is supplied to an energy harvesting sensor via an external energy harvesting process which is then either immediately used for providing a feedback signal to a plant via a wireless channel or stored for later use in a finite-capacity battery.
		}
		\label{fig:arch}
	\end{figure}
	
	\subsection{Plant Dynamics}
	\label{subsec:plant}
	We consider plants modeled as a switched linear system
	\begin{equation}
	\label{eq:sys}
	\begin{aligned}
	x(t+1) = 
	\begin{cases}
	A_c x(t) + \noise(t), & \gamma(t) = 1;\\
	A_o x(t) + \noise(t), & \gamma(t) = 0;\\
	\end{cases}
	\end{aligned}
	\end{equation}
	where the random variable $\gamma(t)$ indicates whether or not the plant has received a feedback signal at time $t,$ $A_c$ is a real $(n \times n)$-dimensional matrix which describes the nominal evolution of the state $x(t)$ when the system operates in closed loop (i.e. has successfully received a feedback signal), $A_o$ is an $(n \times n)$-dimensional real matrix which gives the nominal evolution of the state variable $x(t)$ in open loop (i.e. when the system has not received a feedback signal), and $\{\noise(t)\}$ is a sequence of independent, identically distributed (i.i.d.) random variables from a mean-zero distribution with finite, positive definite second moment matrix $W.$  Intuitively, we can think of the system operating in closed loop as applying an \emph{a priori} designed simple linear feedback to the plant.
	
	Note that in the case that $A_o$ is stable, the trivial transmission policy of \emph{never} transmitting feedback stabilizes the evolution of this system.  As such, we expect that in most interesting instances of this problem, $A_o$ will be unstable, though our analysis does not explicitly require this to be the case.  Indeed, it may well be the case that when $A_o$ and $A_c$ are both stable, a switching policy can still be used in order to optimize system performance, in which case our certification method will be useful.  Indeed, it is well known that it is possible to switch between two stable linear systems in such a way so as to induce instability (see, e.g., \cite[Example 3.17]{Costa2005}), and so closed-loop stability must be certified explicitly.
	
	Note also that while the disturbance process $\{\noise(t)\}$ is considered to be i.i.d. and mean-zero, results similar to those which we demonstrate here hold in the case where $\{\noise(t)\}$ is neither i.i.d. nor mean-zero, but consideration of such cases significantly complicates the underlying analysis, and is thus left for formal discussion in future work.  The assumption that $\{\noise(t)\}$ has square integrable increments is essential, as without such an assumption, the second moment of the plant state process $\{x(t)\}$ becomes undefined after only one time step.  We do not expect these to be severe limitations in practice, as many common disturbance models (e.g. i.i.d. Gaussian) satisfy these assumptions.
	
	\subsection{Communication Channel}
	\label{subsec:channel}
	In order to model channel imperfections and the decision process involved in determining when to transmit a signal, we model the distribution of $\gamma(t)$ as itself being a function of the amount of energy committed by the sensor to transmitting the feedback signal at time $t.$  We interact with the behavior of $\{\gamma(t)\}$ by selecting the sensor transmission energy $\Energy(t)$ at each time $t$, where the selection may in general be stochastic, in which case we design its distribution. The probability that the plant successfully receives the communication conditioned on a particular transmission energy $\energy$  is given by 
	\begin{equation}
	\Pr(\gamma(t) = 1 \, | \, \Energy(t) = \energy) =
	\begin{cases}
	\lambda, & \energy \geq \energythresh;\\
	0, & \texttt{otherwise};
	\end{cases} 
	\end{equation}
	where $\energythresh$ is an energy threshold above which the transmission is successful with probability $\lambda,$ and below which all transmissions are unsuccessful.  This model well approximates stochastic channels, such as the sigmoid models often considered in practice~\cite{Gatsis2014, Zhang2008, Ploplys2004}.
	
	\subsection{Harvesting Processes}
	\label{subsec:harvesting}
	We assume the energy harvesting process, i.e. the process which details how much energy is received by the sensor from the environment at each time, takes values in some finite set of integers $\HarvestSpace \triangleq [\Hmax]_0,$ and is a known, deterministic function of an observed discrete-time, discrete-space Markov process which is independent of the other stochastic processes in the system model.  Formally, we can decompose $\{H(t)\}$ into three fundamental components: a discrete, finite set $\mathcal{L}$ of latent process states, an $(|\mathcal{L}| \times |\mathcal{L}|)$-dimensional probability transition matrix $\Ltrans,$ and a deterministic function $h: \mathcal{L} \mapsto \HarvestSpace$ which maps an element $\ell$ from the latent space of $\{H(t)\}$ to the range space of $\{H(t)\}.$  Since we may take $\mathcal{L} = \{1, 2,\dots,|\mathcal{L}|\}$ without loss of generality, we characterize energy harvesting processes $\{H(t)\}$ in the remainder of the paper by specifying the pair $(h,\Ltrans).$ 
	
	\begin{figure}[!t]
		\centering \large
		\includegraphics[width=0.5\textwidth]{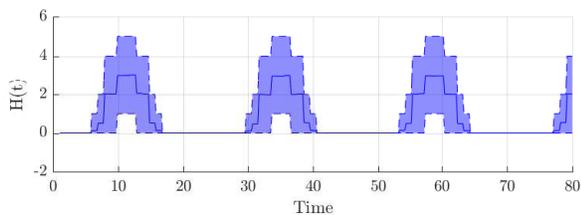}
		\caption{A plot of a $10 \text{k}$ sample Monte Carlo simulation of the harvesting process of solar intensity studied in Section \ref{subsec:periodic}.  The $80\%$ confidence interval is given in blue shade.} \label{fig:solar_source}
	\end{figure}
	
	We make no explicit assumptions about the ergodicity or time-invariance of $\{H(t)\},$ and as such our model incorporates as special cases models which range from simple (e.g., with each $H(t)$ taking the value of some fixed constant $h \in \HarvestSpace$) to complicated (e.g., a function of a periodic Markov process).  This level of abstraction allows us to incorporate models for a wide variety of sources into the same framework.  For instance, we can think of systems subjected to a regular charging cycle as being modeled by a harvesting process which is essentially deterministic (such as the inductive charging used for \emph{in vivo} biomedical sensors \cite{Zhang2009}), as well as systems subjected to highly stochastic, time-varying charging (such as the evolution of solar intensity, as in Figure \ref{fig:solar_source}), with each having an energy source well-modeled by a function of a Markov process.  We discuss this matter in more depth in Section \ref{sec:examples}.
	
	\subsection{Battery Dynamics}
	\label{subsec:battery}
	We assume the battery storage process $\{B(t)\}$ to be bounded above by a finite battery capacity constant $\Bcap,$ which is taken to be a given constant in this paper, but can be designed efficiently if the underlying plant model is sufficiently simple (see \cite{Watkins2017b} for further discussion on this matter).  Any energy available at time $t$ that is not used to transmit a feedback signal and is above the battery capacity is lost due to overflow.  This model guarantees that each element $B(t)$ is in the set $[\Bcap]_0,$ and that the process $\{B(t)\}$ obeys the nonlinear, stochastic dynamics
	\begin{equation}
	\label{eq:battery_propagation}
	B(t+1) = \lproj B(t) + H(t) - \Energy(t) \rproj_{0}^{\Bcap},
	\end{equation}
	where $\Energy(t)$ is the selected sensor transmission energy (see Section~\ref{subsec:channel}), and $H(t)$ is the amount of energy harvested at time $t$ (see Section~\ref{subsec:harvesting}).
	This battery evolution model is common in energy harvesting literature (see, e.g., \cite{Nourian2014b}).  To ensure the sensor always only uses energy available to it at each time, we assume that the energy usage process $\{\Energy(t)\}$ is subject to the energy availability constraint
	\begin{equation}
	\label{ineq:energy_causality}
	\Energy(t) \leq B(t) + H(t),
	\end{equation}
	which guarantees that the system never uses energy in excess of the current stored energy, and the amount of energy which has been harvested in the current time increment.  Note that this constraint implies directly that $\Energy(t)$ takes values on the set $[\Hmax + \Bcap]_0$ at all times.
	
	\subsection{Transmission History}
	\label{subsec:transmit_hist}
	
	We study sensors which have an indicator in memory, storing whether or not the sensor has attempted to transmit feedback in the previous time slot.  We denote this value at time $t$ by $F(t),$ which evolves as
	\begin{equation}
	\label{eq:hist_dyn}
	F(t+1) = 
	\begin{cases}
	1, & \Energy(t) \geq \energythresh;\\
	0, & \text{otherwise}.
	\end{cases}
	\end{equation}
	In Section \ref{sec:principles_policy}, we show that $F(t)$ is not necessary to design stabilizing transmission policies for systems with scalar plant dynamics, but that it helps significantly for designing policies for system with higher-dimensional plants.
	
	\subsection{Transmission Policies}
	\label{subsec:transmission}
	In this paper, the evolution of the plant is controlled by a given \emph{transmission policy}, i.e. a stochastic decision rule that the sensor uses to determine when and how to dedicate energy to providing feedback to the plant.  We can think of transmission policies as conditional probability distributions, and we restrict them to be conditioned only on the current battery level, latent harvesting state, and transmission history value.  Notationally, we have that our policy $u$ is defined as 
	\begin{equation}
	\label{eq:control_def}
	\begin{aligned}
	&u(\energy|B(t),L(t),F(t)) \\
	&\hspace{60 pt} \triangleq \Pr(\Energy(t) = \energy \, | \, B(t), L(t), F(t)).
	\end{aligned}
	\end{equation}
	Note that our assumption that the implemented transmission policy is conditioned as such is somewhat restrictive.  Indeed, such an assumption in effect restricts our consideration to the case in which transmission policies are \emph{memoryless}, as $(B(t),L(t),F(t))$ can be thought of as information which summarizes the current state of the sensor/energy harvesting process pair.  However, we show in Section \ref{subsec:scalar_policy} that insofar as stabilizability is concerned, this assumption is not restrictive for systems with scalar plant dynamics, as a simple greedy memoryless policy suffices to stabilize all such plants.  Moreover, we show in Section \ref{subsec:heuristic_dwell} that such memoryless policies are general enough to be able to encode predictive dwell time policies, which have features which help to mitigate the complexities associated to controlling systems with general linear dynamics.  Moreover, it is an assumption which plays a critical role in our analysis, and it should not be expected that a substantially more complex case can be handled in full generality, for reasons discussed in more detail in Section \ref{sec:stab_cert}.
	
	\subsection{Problem Statement}
	\label{subsec:problem}
	The main objective of this work is to develop an efficient stability test for the plant state process $\{x_u(t)\},$ where $u$ is a fixed, memoryless transmission policy.  As such, we must formally define a notion of stability for use in this paper.
	\begin{defn}[Mean-Square Exponential Ultimately Bounded]
		\label{def:exponential_mean_square_stability}
		{\rm
			The process $\{x_u(t)\}$ is mean-square exponentially ultimately bounded if and only if there exists some finite constants $\alpha \geq 1,$ $\xi \in (0,1),$ and $M \geq 0$ such that 
			\begin{equation}
			\E\|x_u(t)\|_2^2 \leq \alpha \xi^t \E \|x_u(0)\|^2_2 + M \Tr(W),
			\end{equation}
			holds for all $t \in \Znonnegative,$ and any square integrable $x_u(0).$} \oprocend
	\end{defn}
	Note that we use this definition of stability in order to emphasize the role that disturbances play in the evolution of the system.  In the case where the disturbances of significant, the trace of $W$ will be significant as well, and the system's ultimate bound may be large.  However, when the disturbances tend towards zero, so too does the system's ultimate bound.  Moreover, it will be shown in Section \ref{sec:stab_cert} that the system under consideration may be modeled as a MJLS, for which it is known that mean-square exponential stability and mean-square asymptotic stability are equivalent \cite[Theorem 3.9]{Costa2005}.  As such, the main results of our text are largely invariant to the notion of stability considered, and we will simply refer to $\{x_u(t)\}$ as stable whenever it satisfies Definition \ref{def:exponential_mean_square_stability}.
	
	To make our technical statements concise, we often make reference to an \emph{energy harvesting control system}, defined as:
	\begin{defn}[Energy Harvesting Control System]
		{\rm
			An energy harvesting control system (EHCS) is formally defined as the $7$-tuple $(A_c, A_o, h,\Ltrans, \lambda, \energythresh, \Bcap),$ which encodes the closed-loop dynamics, open-loop dynamics, energy harvesting process, packet reception probability, transmission energy, and battery capacity of the system, respectively.}\oprocend
	\end{defn} 
	\noindent From the preceding discussion, we see that the object $(A_c, A_o, h,\Ltrans, \lambda, \energythresh, \Bcap)$ contains all parameters of the model proposed, except for the disturbance process.  This is due to the fact that under our assumptions on $\{\noise(t)\},$ the stability of a energy harvesting control system is unaffected by the presence of stochastic disturbances, which we show formally in Proposition \ref{prop:equiv} (see Section \ref{sec:stab_cert}). 
	
	\emph{The problem we consider in this paper} is developing a tractable means for verifying the closed-loop stability of the plant state process $\{x_u(t)\}$ of an EHCS $(A_c, A_o, h, \Ltrans, \lambda, \energythresh, \Bcap),$ under a fixed transmission policy $u,$ which is defined in the sense of Section \ref{subsec:transmission}.  This problem is important insofar is that it allows a system designer to formally certify that under a fixed transmission policy, the evolution of a given plant will remain safe.  We present the solution to this problem in Section \ref{sec:stab_cert}, in which we show that the evolution of an EHCS can be embedded into a MJLS, and adapt a relevant stability result from the MJLS literature to our setting.  We demonstrate the problem considered admits transmission policies and energy harvesting process models which are sufficiently general for useful applications with detailed discussion in Section \ref{sec:principles_policy} and Section \ref{sec:examples}, respectively.
	
	\section{Stability Certification for Energy Harvesting Control Systems}
	\label{sec:stab_cert}
	
	We now develop an efficient test for determining the stability of an EHCS under a given transmission policy $u,$ where $u$ satisfies the definition given in Section \ref{subsec:transmission}.  We develop the test by embedding the evolution of the plant state process into the dynamics of a MJLS, and then adapt a stability result from the MJLS literature to show that the standard MJLS stability test can be used in our setting.  The essence of the embedding we develop is captured by noting that for a fixed policy $u,$ the process $\{S(t) \triangleq (B(t), L(t), \gamma(t), F(t))\}$ is Markovian, and contains everything necessary to model the dynamics of $\{x_{u}(t)\}$ as a MJLS.
	
	The process $\{S(t)\}$ evolves on the state space $\mathcal{S} \triangleq [\Bcap]_0 \times \mathcal{L} \times \{0,1\} \times \{0,1\}.$  We demonstrate that $\{S(t)\}$ is Markovian by verifying that the Markov property holds, i.e. that $S(t+1)$ is independent of $S(t-1),$ when conditioned on $S(t).$  To do so, we note that transition probabilities $$\Pr_{u}(S(t+1) = s^\prime | S(t) = s) \triangleq \MJLSmat{u}$$ are constants determined by the EHCS's specification.  To make this point more concrete, define $b,$ $\ell,$ $\gamma,$ $\energy,$ and $f$ to be the battery level, latent harvesting process state, loop closure state, energy usage, and transmission history value of the system at state $s;$ define $b^\prime,$ $\ell^\prime,$ $\gamma^\prime,$ $\energy^\prime,$ and $f^\prime$ likewise for $s^\prime.$  We now show that the probability of transitioning to some state $s^\prime$ from a particular state $s$ depends on the probabilities in the EHCS specification, and whether the battery and history dynamics are respected.
	
	Since we have assumed $\{L(t)\}$ to be Markovian and independent of the packet drop process, one may check for all transitions from a state $s$ to a state $s^\prime$ that occur with positive probability, we have transition probabilities given by
	\begin{equation}
	\label{eq:psi_def1}
	\begin{aligned}
	\MJLSmat{u} \triangleq
	\begin{cases}
	\lambda u_{s^\prime}^{\{\energy^\prime \geq \energythresh\}} u_{s}^{\energy}  \Ltrans(\ell^\prime,\ell) ,& \gamma^\prime = 1;\\
	((1 - \lambda) u_{s^\prime}^{\{\energy^\prime \geq \energythresh\}} + u_{s^\prime}^{\{\energy^{\prime} < \energythresh \}}) u_{s}^{\energy} \Ltrans(\ell^\prime,\ell), & \gamma^\prime = 0;
	\end{cases}
	\end{aligned}
	\end{equation}
	where $\Ltrans(\ell^\prime,\ell)$ denotes the probability of transition from latent harvesting state $\ell$ to state $\ell^\prime$ (see Section~\ref{subsec:harvesting}), and $u_{s}^{\{\energy \geq \energythresh\}}$ is notational shorthand for the probability that the amount of energy used at state $s$ is greater than or equal to $\energythresh.$  Intuitively, the right hand side of \eqref{eq:psi_def1} partitions the set of possible process transitions $\mathcal{S} \times \mathcal{S}$ with non-zero transition probabilities into events corresponding to the evolution of the states of the battery and loop closure processes.  For all transitions which occur with zero probability due to not adhering to the battery dynamics \eqref{eq:battery_propagation} or the history dynamics \eqref{eq:hist_dyn}, we have $\MJLSmat{u} \triangleq 0.$  Note that $\MJLSmat{u}$ is uniquely defined: for each state $s,$ the transmission policy specifies exactly one probability distribution for $\energy,$ which may then be used to evaluate \eqref{eq:psi_def1} to a particular constant.  As such, it fully specifies the transition probabilities of $\{S(t)\},$ demonstrating that $\{S(t)\}$ is a Markov chain, as claimed.  We now embed the evolution of $\{x_u(t)\}$ in a MJLS using this fact.
	
	By noting that the value of $\gamma(t)$ is embedded in $S(t),$ we may define $\gamma(S(t))$ to be the state of the loop closure variable at $S(t),$ and write the plant state dynamics \eqref{eq:sys} as
	\begin{equation}
	\label{eq:MJLS_dyn}
	x_u(t+1) = A_s x_u(t) + \noise(t),
	\end{equation}
	where we have defined
	\begin{equation}
	A_s \triangleq
	\begin{cases}
	A_c & \gamma(s) = 1;\\
	A_o & \gamma(s) = 0;
	\end{cases}
	\end{equation}
	with $s = S(t).$  From this, it follows that for any fixed, memoryless transmission policy $u,$ the process $\{x_u(t)\}$ is a MJLS with dynamics \eqref{eq:MJLS_dyn}, mode process $\{S(t)\},$ and disturbance process $\{\noise(t)\}.$  We now state the stability test we have established in the following theorem, which shows that the stability of an EHCS under a memoryless transmission policy $u$ may be determined by solving a semidefinite program.
	\begin{theorem}[EHCS Mean-square Stability Test]
		\label{thm:linear_prog}
		Let $\mathcal{S}$ be the state space of the mode transition process $S(t) = \{B(t),\ell(t),\gamma(t),F(t)\}$ generated by an EHCS, and fix some positive constant $D > 0.$  The EHCS is stable under the transmission policy $u$ if and only if the optimal value of the semidefinite program
		\begin{equation}
		\label{prog:sdp}
		\begin{aligned}
		&\underset{v \in \real, \{R_s \in \sympossemidef\}_{s \in \mathcal{S}}}{\text{minimize}} &&v\\
		&&& A_{s}^T \sum_{s^\prime \in S} \MJLSmat{u} R_{s^\prime} A_s  - R_{s} \curlyleq I v, \, \forall s \in \mathcal{S},\\
		&&& v \geq -D;
		\end{aligned}
		\end{equation}
		is equal to $-D,$ where $\MJLSmat{u}$ is defined by \eqref{eq:psi_def1}.
	\end{theorem}
	\begin{proof}
		See Appendix \ref{app:thm:linear_prog}.
	\end{proof}
	Theorem \ref{thm:linear_prog} provides a simple, efficient test for assessing the stability of an EHCS $(A_c, A_o, h, \Ltrans, \lambda, \energythresh, \Bcap)$ under a fixed transmission policy $u.$  Note that the role of $D$ in \eqref{prog:sdp} is simply in bounding the value of $v$ below, which is only important insofar that it theoretically guarantees that standard semidefinite programming algorithms will terminate with a solution in time which is polynomial with respect to the dimensionality of the plant, and the cardinality of $\mathcal{S}$ (see, e.g., \cite{Vandenberghe1996} for a more detailed discussion on the complexity of semidefinite programming).
	
	Intuitively speaking, this method functions by computing a mode-dependent quadratic Lyapunov function $\Lyap(x,s) = x^T R_s x,$ which serves to certify the stability of the system.  The system of linear matrix inequalities in the constraints of the program ensure that the value of the Lyapunov function evaluated on the plant state process of the EHCS is a strict supermartingale, i.e. it enforces that
	\begin{equation*}
	\E_u [x^T(t+1)R_{s(t+1)}x(t+1)| x(t), s(t) ] < x^T(t)R_{s(t)}x(t)
	\end{equation*}
	holds for all times $t.$  That this inequality implies exponential mean-square stability for unperturbed Markov jump linear systems is known (see, e.g. \cite{Costa2005}), however the embedding used to take the EHCS model to a MJLS we have constructed is novel.  However, the particular notion of stability considered here is not covered by standard MJLS stability results, due to the lack of ergodicity of $\{S(t)\}.$  Hence, we must demonstrate that testing for stability of the unperturbed MJLS is equivalent to testing for stability of the perturbed MJLS.  As we need this equivalence result again later in the text (see Section \ref{sec:principles_policy}), we state it here in the following proposition.
	
	\begin{prop}[Equivalance of EMSUB and EMSS]
		\label{prop:equiv}
		Consider the dynamical system
		\begin{equation}
		\label{eq:z_dyn}
		z(t+1) =
		\begin{cases}
		A_c z(t), \gamma(t) = 1;\\
		A_o z(t), \gamma(t) = 0;
		\end{cases}
		\end{equation}
		where the random variable $\gamma(t)$ indicates whether or not the plant has received a feedback signal at time $t,$ as in the description of \eqref{eq:sys}.  Let $\{z_u(t)\}$ be the stochastic process generated by applying the transmission policy $u$ to \eqref{eq:z_dyn}, and define $\{x_u(t)\}$ similarly.  The process $\{x_u(t)\}$ is exponentially mean-square ultimately bounded if and only if the process $\{z_u(t)\}$ is exponentially mean-square stable.
		
		Moreover, if $\{z_u(t)\}$ is exponentially mean-square stable with constants $\alpha_u,$ and $\xi_u,$ i.e.
		\begin{equation}
		\E[z_u^T(t) z_u(t)] \leq \alpha_u\E[z_u^T(0) z_u(0)] \xi_u^t,
		\end{equation}
		holds for all time $t \geq 0,$ it then follows that we may choose $M_u = \frac{\alpha_u}{1 - \xi_u}$ to verity the exponential mean-square ultimate boundedness of $\{x_u(t)\}.$
	\end{prop}
	
	\begin{proof}
		See Appendix \ref{app:prop:equiv}. 
	\end{proof}
	
	In light of Proposition \ref{prop:equiv}, formal proof of Theorem \ref{thm:linear_prog} is straightforward.  In particular, we may adapt known results for the exponential mean-square stability of unperturbed MJLS to our framework, and verify that the solution of the particular semidefinite program given is indeed $-D$ when the system is stable.  For purposes of completeness, this argument is given in detail in Appendix \ref{app:thm:linear_prog}.
	
	\begin{remark}[Stability Certification with Fixed Lyapunov Function]
		\label{rem:co_des}
		{\rm
			All of the results of this section are written with the perspective that the transmission policy is fixed, and the Lyapunov function used to certify stability is to be computed.  This was done because there are systems for which one can design good heuristic transmission policies.  In fact, we show in Section \ref{sec:principles_policy} that for scalar systems, one only need to check the stability of a particular greedy transmission policy.  However, it is worth noting that this perspective is inessential to the fundamental theory.  
			
			Computation of a stability certificate is also tractable when the matrices which define the Lyapunov function $\Lyap(x,s) = x^T R_s x,$ i.e. $\{R_s\}_{s \in \mathcal{S}},$ are fixed, but the transmission policy $u$ is left to be determined.  To do so, one may use a similar program to that of \eqref{prog:sdp}, but in which the variable matrices $\{R_s\}_{s \in \mathcal{S}}$ are taken to be program data, and the constants used to define $u$ are made to be optimization variables, subject to the constraints that $u$ is transmission policy in the sense of Section \ref{subsec:transmission}.  Having \emph{both} the transmission policy and the Lyapunov function left as unknowns makes the constraints of the optimization problem a set of bilinear matrix inequalities which are nonconvex, and in general difficult to solve.
			\oprocend}
	\end{remark}
	
	\begin{remark}[Stability Certification for General Policies]
		{\rm
			As noted in Section \ref{subsec:transmission}, we have restricted our attention here to a particular subset of memoryless transmission policies.  While in principle the stability certification test given by Theorem \ref{thm:linear_prog} can be extended in a straightforward manner to any memoryless policy, we will see in Section \ref{sec:principles_policy} that the class of transmission policies studied is sufficiently general to stabilize many interesting systems.  However, this class of policies is technically restrictive in the sense that strictly more general policies can be defined in practice.  For example, one may wish to make the the transmission policy time-varying.
			
			We have not explicitly considered more general policies, because providing an efficient test for certifying the stability of the EHCS under general transmission policies is technically challenging.  In particular, recent results from the MJLS literature \cite{Lun2016} show that determining the mean-square stability of a MJLS under a time-varying mode transition process is $\NP$-hard.  As this is the problem which would be faced if we allowed for time-varying transmission policies, stability certification would be hard for such a problem.  \oprocend
		}
	\end{remark}
	
	\begin{remark}[Alternate Stability Tests]
		{\rm
			It is well known in the MJLS literature that one may check the mean-square stability of the system by determining if the spectral radius of a linear operator induced by the considered MJLS less than one.  This test is equivalent to that which was given by Theorem \ref{thm:linear_prog}, but may be more computationally efficient to check in some circumstances.  We have fully detailed the semi-definite programming stability test here because it allows for a more accessible construction, and highlights the importance of considering fixed policies (see Remark \ref{rem:co_des}).  A person interested in reading further about the spectral radius test can consult standard MJLS references, such as \cite[Chapter 3]{Costa2005}. \oprocend
		}
	\end{remark}
	
	\section{Transmission Policy Design}
	\label{sec:principles_policy}
	In this section, we develop some principles for designing good transmission policies for EHCS.  Since co-designing a transmission policy along with the Lyapunov function required to verify the system's closed-loop stability is a nonconvex optimization problem (see Remark \ref{rem:co_des}), being able to find good transmission policies is essential in designing controllers which certifiably stabilize the evolution of the system.  Formally, we decompose our results into two classifications: those for systems with scalar plants, and those for systems with nonscalar plants.
	
	For scalar systems, we see that one only ever need to check the stability of the system under a particular greedy transmission policy, and that this policy alone certifies the stabilizability of the system, i.e. whether or not a stabilizing causal transmission policy exists.  This is important, insofar that it allows system designers to chose components of scalar EHCS (e.g. the battery capacity $\Bcap$) so as to guarantee the existence of a stabilizing policy; we discuss this topic in greater detail in the preliminary work \cite{Watkins2017b}.  For general nonscalar systems, we show how one can use the concept of dwell time to identify stabilizing polices for systems which the greedy policy fails to stabilize.
	
	\subsection{A Stabilizing Policy for Scalar Systems}
	\label{subsec:scalar_policy}
	In this subsection, we show that a simple, greedy transmission policy is sufficient for stabilizing an EHCS with scalar plant dynamics.  The greedy transmission policy in question is described mathematically by the conditional probability distribution $u_g,$ defined as
	\begin{equation}
	\label{eq:greedy_policy}
	u_g(\energy|B(t),H(t)) \triangleq
	\begin{cases}
	1, & \energy = \energythresh, B(t) + H(t) \geq \energythresh;\\
	1, & \energy = 0, B(t) + H(t) < \energythresh;\\
	0, & \texttt{otherwise}.
	\end{cases}
	\end{equation}
	The transmission policy $u_g$ applies exactly $\energythresh$ units of energy to transmitting a feedback signal at precisely those moments at which the sensor has enough energy available to do so.  Despite its simplicity - note that it is a deterministic function of only the battery and harvesting states at each time, and \emph{not} the transmission history - we prove that it is the \emph{only} policy which needs to be investigated to establish the stabilizability of an EHCS.  That is if \emph{any} causal transmission policy exists which stabilizes a particular EHCS, the greedy transmission policy $u_g$ does so as well.  We formalize this as follows:
	
	\begin{theorem}[Existence of Stabilizing Policy for Scalar EHCS]
		\label{thm:existence_stab}
		Consider an EHCS $(A_c, A_o, h,\Ltrans, \lambda, \energythresh, \Bcap)$ with $n = 1.$  There exists a causal transmission policy
		\begin{equation*}
		\begin{aligned}
		u^*(\energy| \{(S(\tau),x(\tau))\}_{\tau = 0}^{t}) \triangleq \Pr(\Energy(t) = \energy | \{(S(\tau),x(\tau))\}_{\tau = 0}^{t})
		\end{aligned}
		\end{equation*}
		which stabilizes the $EHCS$ if and only if the process $\{x_{u_g}(t)\}$ is stable.  That is, a given EHCS is stabilizable if and only if it is stable under the greedy transmission policy.
	\end{theorem}
	
	We detail next the essential features of the argument supporting Theorem \ref{thm:existence_stab}.  Interestingly, most of the weight of the proof can be shifted onto proving a pathwise stochastic dominance inequality between the greedy policy and any other causal policy.  To demonstrate this, we need to formally define a sample space for the process.  For the remainder of the paper, we define the sample space as 
	\begin{equation*}
	\begin{aligned}
	&\Omega \triangleq \{(f_\omega,g_\omega) \,| \, f_{\omega}:\Znonnegative \mapsto \mathcal{L} \times [0,1], g_{\omega}: \Znonnegative \mapsto \{0,1\}\},
	\end{aligned}
	\end{equation*} 
	where $f_\omega(t)$ is a vector function containing the evolution of $\{L(t)\}$ in its first component $f_{1 \omega}(t)$ and numbers for the randomization required to determine particular actions from a stochastic transmission policy $u$ in its second component $f_{2 \omega}(t),$ and $g_\omega(\ell)$ is a function indicating whether or not the $\ell'$th feedback attempt reaches the plant successfully.  Note that we have defined the sample space to consist of \emph{pairs} of functions so as to be able to index time $t$ and the number of communication attempts $\ell$ made by the system separately; this technical detail is important.  
	
	Intuitively, the function $f_w$ contains all of the randomization needed to model the processes which are indexed naturally with respect to time: from it, we may fully determine the evolution of the harvesting process $\{H(t)\},$  as described in Section~\ref{subsec:harvesting}, and the randomization required to implement a stochastic transmission policy as described in Section~\ref{subsec:transmission}.  The function $g_\omega$ contains the randomization needed to model the communication channel  according to Section~\ref{subsec:channel}: from it, we can determined whether the $k$'th time the transmission energy process exceeds $\energythresh$ - that is, the $k$'th communication \emph{attempt} - results in a successful loop closure.  
	
	The only subtle point required in verifying that $\Omega$ is a proper sample space for the process is in confirming that all process variables are fully determined by the selection of a particular sample path $\omega.$  Since all of the process variables at a particular time $t$ can either be determined directly from $\omega$ or $\Energy(t),$ we briefly discuss how one may compute the channel energy at every time from a selected $\omega.$  After observing any sequence of events through time $t,$ and under any fixed policy $u,$ we may partition $[0,1]$ into a collection of disjoint intervals $\{\Interval_{(u,t)}(\energy)\}$ such that the Lebesgue measure of $\Interval_{(u,t)}(\energy)$ is equal to the probability that $\Energy(t) = \energy$.  By associating to $f_{2 \omega}(t)$ the probability measure of a sequence of i.i.d. uniform random variables on $[0,1],$ we may take $\Energy(t) = \energy$ for whichever $\energy$ satisfies $f_{2 \omega}(t) \in \Interval_{(u,t)}(\energy)$ and have that $\Energy(t)$ follows the correct distribution. 
	
	Note that - unlike in many types of analysis one may perform on models with stochastic control policies - the sample space and probability measure are both unaffected by the particular choice of transmission policy $u.$  This is accomplished by taking the probability measure $\Pr$ on the sample space $\Omega$ to be the product of the probability measures of three independent processes.  In particular, we have $\Pr = \Pr_{f_1} \Pr_{f_2} \Pr_{g},$ where $\Pr_{f_1}$ is the measure induced by the latent-state process of $\{H(t)\},$ $\Pr_{f_2}$ is the measure of a sequence of i.i.d. uniform random variables on the unit interval, and $\Pr_g$ is the measure of a sequence of i.i.d. Bernoulli random variables with success probability $\lambda.$  This ability to decouple the choice of strategy from the choice of probability measure is important in that it allows us to compare the performance of different control policies on a sample-by-sample basis.  More precisely, if we let $N_{u}(t; \omega)$ be the number of successful loop closures attained by a transmission policy $u$ through time $t$ on sample path $\omega,$ we can show the following:  
	
	\begin{lem}[Pathwise Dominance of Greedy Policy]
		\label{lem:pathwise_dominance}
		For all times $t \in \Znonnegative,$ and all samples $\omega \in \Omega,$ it holds that
		\begin{equation}
		N_{u} (t; \omega) \leq N_{u_g} (t; \omega),
		\end{equation}
		i.e. the greedy transmission policy dominates every other causal transmission policy in terms of successful loop closures at every time along every sample path.
	\end{lem}
	\begin{proof}
		See Appendix \ref{app:lem:pathwise_dominance}.
	\end{proof}
	
	A direct consequence of Proposition \ref{prop:equiv} is that proving Theorem \ref{thm:existence_stab} only explicitly requires analyzing the evolution of the undisturbed process $\{z_u(t)\}.$  Considering the consequences of Proposition \ref{lem:pathwise_dominance} in detail demonstrates that for scalar systems, the greedy transmission policy outperforms all others with respect to the state process $\{z_{u}(t)\},$ as stated next.
	
	\begin{cor}[Stochastic Dominance Inequalities]
		\label{cor:z_dom}
		For all times $t \in \Znonnegative,$ and all samples $\omega \in \Omega,$ it holds that
		\begin{equation}
		\label{ineq:z}
		\|z_{u_g}(t; \omega)\|_2^2 \leq \|z_{u}(t; \omega)\|_2^2,
		\end{equation}
		and hence it also holds that
		\begin{equation}
		\label{ineq:ms_dominance}
		\E \|z_{u_g}(t)\|_2^2 \leq \E \|z_{u}(t)\|_2^2,
		\end{equation}
		where the expectation is taken with respect to the probability measure $\Pr = \Pr_{f_{1}}\Pr_{f_{2}}\Pr_{g}.$
	\end{cor}
	
	Since Corollary \ref{cor:z_dom} is proven immediately by noting that $z_u(t)$ is a monotone decreasing function of $N_u(t)$ and integrating, formal proof is not given here.  As a direct consequence of \eqref{ineq:ms_dominance}, it also holds that $\{z_{u}(t)\}$ is exponentially mean-square stable only if $\{z_{u_g}(t)\}$ is exponentially mean-square stable.  Since stability of $\{z_{u_g}(t)\}$ implies stabilizability of the EHCS and Proposition \ref{prop:equiv} implies that $\{x_{u_g}(t)\}$ is stable if $\{z_{u_g}(t)\}$ is exponentially mean-square stable, Theorem \ref{thm:existence_stab} is proven as well.  
	
	Note that, as stated, the result in Theorem \ref{thm:existence_stab} is conservative.  Stability of the greedy policy suffices as a stabilizability test for a larger class of systems than those with scalar plants.  Showing that this is the case for simultaneously diagonalizable systems involves only basic algebraic arguments, and so its proof is left out of this manuscript.  One can also show that greedy policies suffice for stabilizing systems in which the plant matrices $A_c$ and $A_o$ commute, though the argument is more involved.  In particular, the stochastic dominance inequality \eqref{ineq:z} no longer holds for all time.  Indeed, the greedy algorithm can be suboptimal for commuting systems, but the sub-optimality only grows as a polynomial with respect to time, and as such the existence of an exponentially stabilizing policy still implies that the greedy policy exponentially stabilizes the system.  However, formally proving this result requires a long, technical argument with many algebraic details, and discussing it further would take us too far from the main path of the results we wish to present in this text.  As such, we discuss it no further here, and it leave it as future work.
	
	\subsection{Predictive Dwell Time Policies}
	\label{subsec:heuristic_dwell}
	An essential feature of nonscalar linear systems which separates them from scalar linear systems is a lack of multiplicative commutativity.  Without being able to commute the multiplication inherent in detailing the system's dynamics, it is difficult to find a tractable representation of the system's evolution.  In particular, where for scalar systems, we have that $x_u(t)$ is a function of only the number of attained loop closures through time $t$ and the initial condition $x_u(0),$ for general systems, this is not the case.  For emphasis, we write 
	\begin{equation}
	\label{eq:mode_decomp}
	x_u(t + k) = \prod_{j = 1}^{\ModeChangeCount(t,k)} A_{\Mode(j)}^{\ModeTime} x_u(t),
	\end{equation}
	where $\ModeChangeCount(t,k)$ is the number of mode changes experienced by the plant on the interval $[t, t+k],$ $\Mode(j)$ is the mode the system is operating in after the $(j-1)$'st mode change, and $\ModeTime$ is the number of time slots in which the system remains in the $j$'th mode.  Equation \eqref{eq:mode_decomp} demonstrates that the sequence of dwell times $\{\ModeTime\}_{j \in \ModeChangeCount(t,k)}$ are central in defining the map which takes $x_u(t)$ to $x_u(t + k).$  As such, they play a central role in our considerations in this section.  
	
	We now define a class of transmission policies informed by the decomposition \eqref{eq:mode_decomp}, insofar that it emphasizes the importance of keeping the system in the closed-loop mode for significant periods of time without interruption.  Essentially, dwelling in the stable mode for a long period of time allows the system to overcome the possibility of polynomial growth introduced by the switching, by way of allowing the exponential decay induced by the stability of the closed-loop matrix to have enough time to dominate the evolution of the plant.  We refer to the policies we study as predictive dwell time transmission policies, defined as follows: 
	
	\begin{defn}[Predictive Dwell Time Policy]
		\label{defn:dwell_pol}
		Choose some desired dwell time $k,$ and some probability $p \in [0,1].$  We define the predictive dwell time transmission policy with parameters $k$ and $p$ to be the policy which uses exactly $\energythresh$ units of energy at the first moment which the sensor detects that it can attempt $k$ consecutive loop closures with at least probability $p,$ and each moment thereafter until it lacks sufficient energy to do so any longer.  
		
		That is, we define the predictive dwell time policy with parameters $k$ and $p$ as the memoryless transmission policy
		\begin{equation}
		\begin{aligned}
		&u(\energy | B(t), L(t), F(t)) \triangleq \\
		&\begin{cases}
		1, & \energy = \energythresh, B(t) + h(L(t)) \geq \energythresh,F(t) = 0, \phi_{k,s} \geq p;\\
		1, & \energy = 0, B(t) + h(L(t)) \geq \energythresh, F(t) = 0, \phi_{k,s} < p;\\
		1, & \energy = \energythresh, B(t) + h(L(t)) \geq \energythresh, F(t) = 1;\\
		1, & \energy = 0, B(t) + h(L(t)) < \energythresh;\\
		0, & \texttt{otherwise}.
		\end{cases}
		\end{aligned}
		\end{equation}
		where we have defined the symbol $\phi_{k,s}$ for the probability that the process will have sufficient energy available to attempt $k$ consecutive loop closures, given an initial process state $s,$ and that the system will use exactly $\energythresh$ units of energy at each time in order to do so, i.e.
		\begin{equation}
		\label{eq:dwell_time_prob}
		\phi_{k,s} \triangleq \Pr(\cap_{\tau= 0}^{k - 1} \{\mathcal{B}(\tau) + h(L(\tau)) \geq \energythresh\}| s),
		\end{equation}
		where $\mathcal{B}(\tau)$ represents the battery charge state under the policy $\tau$ units of time into the future, i.e. 
		\begin{equation*}
		\mathcal{B}(\tau) \triangleq 
		\begin{cases}
		b, &\tau = 0;\\
		\lproj \mathcal{B}(\tau-1) + h(L(\tau-1)) - \energythresh \rproj_{0}^{\Bcap}, & \tau > 0.
		\end{cases}
		\end{equation*} \oprocend
	\end{defn}
	
	This policy, in particular, is one of interest in that it includes the greedy transmission policy developed in Section \ref{subsec:scalar_policy} in the special case $k = 1.$  More importantly, it is not obvious from inspection of Definition \ref{defn:dwell_pol} that predictive dwell time policies can be computed efficiently for arbitrary problems.  Specifically, a \naive method of computing the dwell time probabilities $\{\phi_{k,s}\}$ would enumerate all possible sample paths of the embedded state-space process $\mathcal{S}$ over the interval $[t,t+k-1],$ and compute the required probability by summing over the set of samples which have the sensor transmitting feedback information on $k$ consecutive time slots.  The complexity of such enumeration is exponential in $k,$ and as such, would be inefficient.  As such, it is important that we verify that such policies can be computed efficiently, if we are to think of them as useful.  We now show how dynamic programming can be used in order to do so.
	
	In this spirit, define the process $\{G(t) = (B(t), L(t))\},$ with transition probabilities given by
	\begin{equation}
	\label{eq:g_trans}
	\begin{aligned}
	\Gtrans(g^\prime,g) &\triangleq \Pr(B(t) = b^\prime, L(t) = \ell^\prime | B(t) = b, L(t) = \ell),\\
	&= \mathbbm{1}_{\{b^{\prime} = \lproj b + h(\ell) - \energythresh \rproj_{0}^{\Bcap} \}} \Ltrans(\ell^\prime, \ell),
	\end{aligned}
	\end{equation}
	and define the set $\Target$ as the subset of states of $G$ such that the system has enough energy to transmit feedback, i.e.
	\begin{equation}
	\Target \triangleq \{(b,\ell) \, | \, b + h(\ell) \geq \energythresh \}.
	\end{equation}
	Defining $\MembershipEvent{t}$ to be the event that the process $G$ is in $T$ at time $t,$ i.e. $\MembershipEvent{t} \triangleq \{G(t) \in T\},$ it suffices to demonstrate that we can compute the probability $\Pr(\cap_{t = 0}^{k-1} \MembershipEvent{t}| G(0))$ tractably with respect to the planning horizon $k.$   By applying the chain rule of conditional probability (see, e.g., \cite[Chapter 2]{Venkatesh2013}), we have
	\begin{equation}
	\label{eq:update_prob}
	\begin{aligned}
	\Pr_{G(0)}(\cap_{t = 0}^{k-1} \MembershipEvent{t}) = \Pr_{G(0)}(\MembershipEvent{k-1} |\cap_{t = 0}^{k-2} \MembershipEvent{t})\Pr_{G(0)}(\cap_{t = 0}^{k-2} \MembershipEvent{t}).
	\end{aligned}
	\end{equation}
	Supposing that we have the value $\Pr_{G(0)}(\cap_{t = 0}^{k-2} \MembershipEvent{t})$ stored in a dynamic programming table, we need only to demonstrate that $\Pr_{G(0)}(\MembershipEvent{k-1}|\cap_{t = 0}^{k-2} \MembershipEvent{t})$ can be computed efficiently.  By decomposing the event $\MembershipEvent{k-1},$ we get
	\begin{equation}
	\label{eq:aggregate_set}
	\Pr_{G(0)}(\MembershipEvent{k-1} |\cap_{t = 0}^{k-2} \MembershipEvent{t}) = \sum_{g \in T} \Pr_{G(0)}(G(k-1) = g |\cap_{t = 0}^{k-2} \MembershipEvent{t})
	\end{equation}
	holds.  By applying Bayes' rule to $\Pr_{G(0)}(G(k-1) = g |\cap_{t = 0}^{k-2} \MembershipEvent{t}),$ we have the identity
	\begin{equation}
	\label{eq:compute_marg}
	\begin{aligned}
	&\Pr_{G(0)}(G(k-1) = g |\cap_{t = 0}^{k-2} \MembershipEvent{t}) = \\
	&\hspace{20 pt} \frac{\Pr_{G(0)}(G(k-1) = g,\MembershipEvent{k-2} | \cap_{t = 0}^{k-3} \MembershipEvent{t}) \Pr_{G(0)}(\cap_{t = 0}^{k-3} \MembershipEvent{t}) }{\Pr_{G(0)}(\cap_{t = 0}^{k-2} \MembershipEvent{t}) },
	\end{aligned}
	\end{equation}
	in which we note that $\Pr_{G(0)}(G(k-1) = g,\MembershipEvent{k-2} | \cap_{t = 0}^{k-3} \MembershipEvent{k})$ is the only term which has not yet been explicitly computed and stored.  We address this by applying the chain rule a final time, to arrive at the identity
	\begin{equation}
	\label{eq:compute_joint}
	\begin{aligned}
	&\Pr_{G(0)}(G(k-1) = g,\MembershipEvent{k-2} | \cap_{t = 0}^{k-3} \MembershipEvent{k}) = \\
	& \hspace{10 pt} \sum_{g^\prime \in T} \Gtrans(g,g^\prime) \Pr_{G(0)}(G(k-2) =  g^\prime | \cap_{t = 0}^{k-3} \MembershipEvent{k}),
	\end{aligned}
	\end{equation}
	which depends only explicitly on the transition matrix $\Gtrans,$ and the terms $\Pr(G(k-2) =  g^\prime | \cap_{t = 0}^{k-3} \MembershipEvent{t}, G(0)),$ which we have explicitly computed in the calculations for $\Pr_{G(0)}(\cap_{t = 0}^{k-2} \MembershipEvent{k}).$  Algorithm \ref{alg:prob_comp} summarizes computing $\Pr_{G(0)}(\cap_{t = 0}^{k-1} \MembershipEvent{t})$ by the method just described.  By inspecting our argument, we have that for each fixed value $k,$ we have $O(|\mathcal{S}|^2)$ computations.  Hence, the total complexity of computing $\Pr_{G(0)}(\cap_{t = 0}^{k-1} \MembershipEvent{t})$ by Algorithm \ref{alg:prob_comp} is $O(k|\mathcal{S}|^2).$
	
	\begin{algorithm}
		Initialization:
		\begin{algorithmic}[1]
			\item Define the process $\{G(t) = (B(t),L(t))\}$ as in \eqref{eq:g_trans};
			\item Define the set $\Target =  \{(b,\ell) \, | \, b + h(\ell) \geq \energythresh \};$
			\item Define the event $\MembershipEvent{t} = \{G(t) \in \Target \};$
			\item Compute $\Pr_{G(0)}(\MembershipEvent{0}) = \mathbbm{1}_{\{G(0) \in T\}};$
			\item Compute $\Pr_{G(0)}(G(1) = g^\prime) = \Gtrans(g^\prime,G(0));$
			\item Compute $\Pr_{G(0)}(\MembershipEvent{1}) = \sum_{g^\prime \in T} \Pr_{G(0)}(G_1 = g^\prime).$
		\end{algorithmic}
		
		For $j \in [2,k-1]:$
		\begin{algorithmic}[1]
			\item Compute $\Pr_{G(0)}(G(t) = g |\cap_{\tau = 0}^{j - 1} \MembershipEvent{\tau})$ by \eqref{eq:compute_joint} and \eqref{eq:compute_marg}.
			\item Compute $\Pr_{G(0)}(\MembershipEvent{j} | \cap_{\tau = 0}^{j-1} \MembershipEvent{\tau})$ by \eqref{eq:aggregate_set}.
			\item Compute $\Pr_{G(0)}(\cap_{\tau = 0}^{j} \MembershipEvent{\tau})$ by \eqref{eq:update_prob}.
		\end{algorithmic}	
		
		\caption{Set Membership Probability Computation} \label{alg:prob_comp}
	\end{algorithm}
	
	Since for each $s,$ we need to compute a value of $\phi_{k,s},$ and the size of $\mathcal{S}$ grows linearly with the size of each component of the process $\{S(t)\},$ we have that for a fixed $k$ and $p,$ we may efficiently compute the predictive dwell time transmission policy.  We record this in the following theorem.
	
	\begin{theorem}[Dwell Time Probability Computation Complexity]
		\label{thm:dwell_comp}
		The worst-case computational complexity of computing the dwell time probabilities $\{\phi_{k,s}\}_{s \in \mathcal{S}}$ required to compute the predictive dwell time transmission policy given in Definition \ref{defn:dwell_pol} for specified parameters $k$ and $p$ is $O(k |\mathcal{S}|^3),$ and can be attained by implementing Algorithm \ref{alg:prob_comp}. 
	\end{theorem} 
	
	We close this section by constructing a minimal example demonstrating that the greedy policy outlined in Section \ref{subsec:scalar_policy} does not suffice to stabilize all possible non-scalar energy harvesting control systems.  Choose $n = 2,$ and let
	\begin{equation*}
	A_c = 
	\left[\begin{matrix}
	0.093 & 0.558 \\
	0.558 & 0.186
	\end{matrix} \right], \,
	A_o = 
	\left[
	\begin{matrix}
	1.050 & 1.000\\
	0.000 & 1.000
	\end{matrix}
	\right ],	
	\end{equation*}
	take the harvesting process $H$ to have the latent space process with transition matrix
	\begin{equation*}
	\Ltrans = \left[
	\begin{matrix}
	0.01 & 0.99\\
	0.99 & 0.01
	\end{matrix}
	\right],
	\end{equation*}
	and energy function $h = \ell - 1,$ and let the packet reception probability be $\lambda = 0.98.$  The reader may verify that while $A_c$ is stable, its operator norm is strictly greater than one, which means there are vectors which grow in Euclidean norm when left-multiplied by $A_c.$  As such, we may expect that we need to link together more than one loop closure in order to induce decay, and as such there may be energy sources for which a predictive dwell time policy may stabilize the system, where the greedy policy does not.  A comparison of the evolution of the system under the greedy transmission policy and the predictive dwell time policy with parameters $k=2$ and $p =0.5$ is given in Figure \ref{fig:dwell_pol_comp}.
	
	Despite outperforming the predictive dwell time policy in terms of the total number of attained loop closures in the simulated time interval, the greedy policy under-performs the dwell time policy in terms of stability.  This example summarizes the essential difficulty in finding stabilizing policies for EHCS with nonscalar plants: optimizing the number of loop closures is not enough to guarantee stability if the plant's dynamics are nonscalar.  However, accounting for the interplay between the modes of the system in the policy design can help mitigate the difficulties encountered in creating a good transmission policy.
	
	\begin{figure}
		\begin{subfigure}[b]{\linewidth} 
			\centering\large 
			\includegraphics[width=\textwidth]{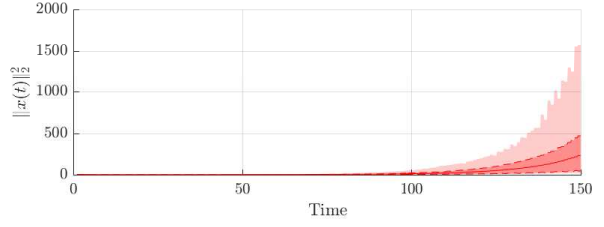}
			\caption{Plant state evolution with greedy policy.}
			\label{fig:high_dim_greed} 
		\end{subfigure}
		\begin{subfigure}[b]{\linewidth} 
			\centering\large 
			\includegraphics[width=\textwidth]{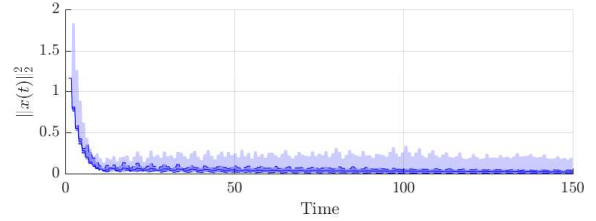}
			\caption{Plant state evolution with the predictive dwell time policy.}
			\label{fig:high_dim_dwell} 
		\end{subfigure} 
		\caption{\small  A plot of a $10\text{k}$ sample Monte Carlo simulation of the example EHCS, comparing the evolution of the plant state of the process with the dwell time transmission policy against the evolution of the process with the greedy transmission policy.  The mean trajectory given by a solid line, the $80\%$ confidence interval in dark shade, with the $98\%$ confidence interval given in lighter shade.  It is clear that the dwell time policy stabilizes the system, whereas the greedy policy does not.} \label{fig:dwell_pol_comp}
	\end{figure}
	
	\begin{remark}[Searching for Stabilizing Policies]
		{\rm
			As noted in Remark \ref{rem:co_des}, searching for a stabilizing memoryless transmission policy without fixing a particular Lyapunov function is a nonconvex optimization problem, and so in general may be difficult to solve.  However, one \emph{can} efficiently search over the set of all predictive dwell time policies, up to a fixed maximum desired dwell time $K_{\max}.$  That is, the number of distinct predictive dwell time policies grows polynomially with respect to the number of states in $\mathcal{S},$ and $K_{\max},$ and so each such policy can be tested individually to determine if any such policy stabilizes the system.  Formal proof of this fact follows from noting that for any fixed $k,$ the particular transmission policy is fully determined by which subset of states begins a transmission sequence.  This partitions the unit interval into at most $|\mathcal{S}| + 1$ disjoint subintervals, where for all $p$ in a particular subinterval, the induced policy is the same.  As in general, there is no guarantee that this class of policies \emph{must} contain a stabilizing policy (as is the case where $n = 1$), we do not dedicate more space to formalizing this concept in greater detail here. \oprocend
		}
	\end{remark}
	
	\section{Example Energy Harvesting Source Models}
	\label{sec:examples}
	
	We now detail how energy harvesting sources can be modeled within the mathematical framework defined in Section \ref{subsec:harvesting}.  While we have in general placed no assumptions on the transition matrix $\Ltrans$ other than it be column stochastic, the statistical models we study in this section all have the following block structure:
	\begin{equation}
	\label{eq:block_struc}
	\Ltrans = 
	\left(
	\begin{matrix}
	0      & 0      & \dots  & 0 & Q_{\Pd} \\
	Q_1    & 0      & \ddots & \vdots & 0     \\
	0      & Q_2    & \ddots & \vdots & 0\\
	\vdots & \vdots & \ddots &  0 & \vdots\\
	0      & 0  & \dots  &    Q_{\Pd-1} & 0
	\end{matrix} \right),
	\end{equation}
	where for each $j$ in $[\Pd],$ we have that $Q_j$ is an $n_{j+1} \times n_{j}$ column stochastic matrix.  By considering \eqref{eq:block_struc} in detail, one may note that the transitions of the encoded Markov process are such that states $1$ through $n_1$ transition exclusively to states $(n_1 + 1)$ through $n_1 + n_2,$ and so forth.  This naturally encodes a time-varying process which is $\Pd$-periodic: any particular state may only be revisited by the process after some multiple of $\Pd$ time slots have passed since its last visit.  We see in the following subsections how this structure allows a user to encode sources which are time-varying.
	
	\subsection{Deterministic Energy Harvesting Sources}
	\label{subsec:deterministic}
	In this subsection, we construct a general model for a situation in which energy arrives at the sensor according to a deterministic, periodic schedule.  This model is appropriate in the case where the end user re-charges the device according to a fixed schedule.  This is the case in some interesting applications, such as \emph{in vivo} biological sensors, which can be recharged by an end-user via an inductive source \cite{Zhang2009}. 
	
	Consider a periodic source with period length $\Pd,$ defined by the periodic sequence $\{c_\tau\}_{\tau = 1}^{\Pd}.$  We may define the sequence $\{c_\tau\}_{\tau = 1}^{\Pd}$ by specifying its first $\Pd$ elements, and then taking $c_t = c_{\tau}$ for whichever $\tau$ is the unique integer less than or equal to $\Pd$ which satisfies the equality $t = \tau + j\Pd$ for some natural number $j.$  The latent space of the energy harvesting process is used for the purpose of indexing time, hence
	\begin{equation}
	\Ltrans 
	= 
	\left(
	\begin{matrix}
	0 & 1\\
	I_{\Pd-1} & 0
	\end{matrix}
	\right),
	\end{equation}
	which we note to be a special case of \eqref{eq:block_struc}, where $Q_j = 1$ for all $j$ in $[\Pd].$  Note that by design, this particular choice of $\Ltrans$ is a permutation matrix, which induces the latent state variable process $\{\ell(t)\}$ to follow the dynamics
	\begin{equation}
	\ell(t+1) = \mod(\ell(t),\Pd) + 1,
	\end{equation}
	which serves to increment $\ell(t)$ around the ring $\{1,2,\dots,\Pd,1,\dots\}.$  By defining $h(\ell) = c_{\ell},$ we then have that the energy harvesting process with $(\Ltrans,h)$ exactly models a source with fixed, deterministic, periodic increments $\{c_{\tau}\}_{\tau = 1}^{\Pd}.$  
	
	As a particular example, we consider a source in which $\Pd = 24,$ $c_\tau = 5$ for all $\tau \in \{1 + 24j\}_{j \in \Nat}$ and $c_\tau = 0$ for all $\tau \notin \{1 + 24j\}_{j \in \Nat}.$  This models the situation in which an end user deterministically recharges the sensor once every day, and endows it with $5$ units of energy.  Using the algorithmic techniques developed in \cite{Watkins2017b} with $A_c = 0.8,$ $A_o = 1.1,$ $\energythresh = 2,$ and $\lambda = 0.98$ we find that the critical battery capacity - that is, the minimum battery capacity required to guarantee the existence of a stabilizing transmission policy - is $2.$  
	
	The results of a simulation of this system under the closed-loop evolution of the system under the greedy transmission policy is given in Figure \ref{fig:det_sim}, where the disturbance process $\{\disturbance(t)\}$ is a sequence of uniform random variables on $[-0.5,0.5].$  Note that the periodicity of the energy harvesting source is apparent in the statistics of the state process.  After each recharging event occurs, the norm of the state decays exponentially quickly.  Between recharging events, the norm increases steadily.  In the case where the battery is larger than the critical battery capacity, the plant state process is stable; in the other, it is not.  Note that the periodicity in the processes' evolution may not be optimal performance, which suggests that an interesting line of future research may be in designing transmission policies which are guaranteed to be stable, but also optimize system performance.
	\begin{figure}
		\begin{subfigure}[b]{\linewidth} 
			\centering\large 
			\includegraphics[width=\textwidth]{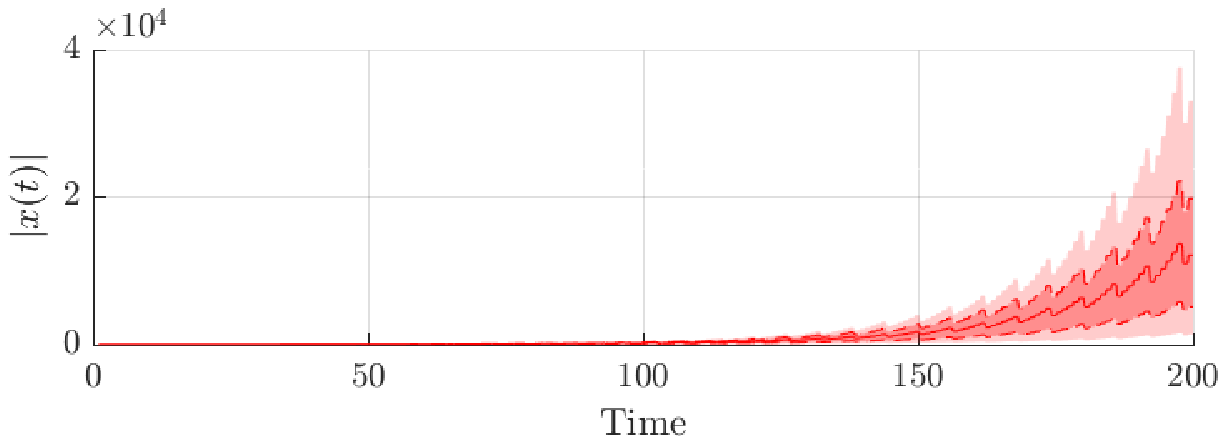}
			\caption{Battery Capacity $\Bcap = \Bcrit-1 = 1.$}
			\label{fig:1a} 
		\end{subfigure} 
		\begin{subfigure}[b]{\linewidth} 
			\centering\large 
			\includegraphics[width=\textwidth]{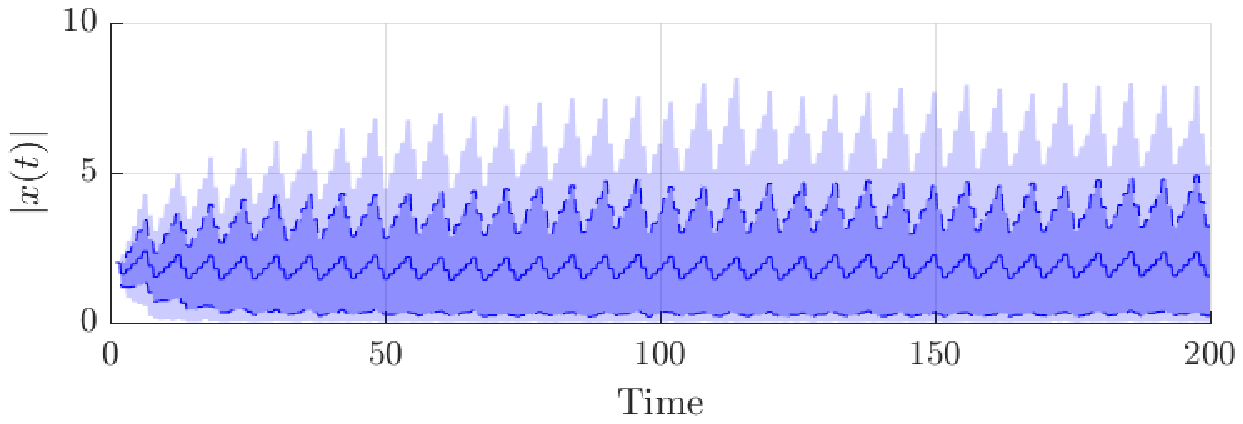}
			\caption{Battery Capacity $\Bcap = \Bcrit = 2.$}
			\label{fig:1b} 
		\end{subfigure}
		\caption{\small  A plot of a $10\text{k}$ sample Monte Carlo simulation of the example EHCS, with the mean trajectory given by a solid line, the $80\%$ confidence interval in dark shade, with the $98\%$ confidence interval given in lighter shade.} \label{fig:state_evo}
		\caption{A simulated EHCS with a deterministic recharging source, as detailed in Section \ref{subsec:deterministic}.  Note that the simulations display periodicity in the statistics of the process, which are due to the periodicity of the source bleeding onto the plant state dynamics through the greedy transmission policy.} \label{fig:det_sim}
	\end{figure}
	
	\subsection{Ergodic Energy Harvesting Sources}
	\label{subsec:ergodic}
	
	In this subsection, we show how the proposed model for energy harvesting sources can be used to model ergodic energy sources.  These can be useful in several application areas.  For example, several works propose ergodic Markov chains as a good model for the dynamics of the intensity of wind speed \cite{Carpinone2015,Tang2015,Xie2017}.  As such, if the sensor is supplied with energy by a small-scale wind turbine, we may expect an ergodic Markov chain to be a good statistical model for the energy harvesting process.
	
	Note that any finite-state ergodic Markov process can be represented as a Markov chain with a finite, irreducible transition matrix (see, e.g., \cite{Stroock2005}).  As a particular example, we consider the EHCS with $A_c = 0.95,$ $A_o = 1.02,$ $\lambda = 0.98,$ $\energythresh = 4,$ $\{\disturbance(t)\}$ as a process of i.i.d. uniform random variables on $[-0.5,0.5],$ and $\{H(t)\}$ as a process with latent Markov process $\{L(t)\}$ with transition matrix
	\begin{equation*}
	\Ltrans = 
	\left(
	\begin{matrix}
	0.70 & 0.15 & 0.00 & 0.00\\
	0.30 & 0.70 & 0.15 & 0.00\\
	0.00 & 0.15 & 0.70 & 0.30\\
	0.00 & 0.00 & 0.15 & 0.70
	\end{matrix} \right)_,
	\end{equation*}
	and energy function $h(\ell) = \ell-1.$  Note that this particular choice of $\Ltrans$ corresponds to the case in which $n = n_1,$ as we may simply take $\Ltrans = S_1.$  As constituted, $\{H(t)\}$ is a skip-free random walk on $[\Hmax]_0,$ and can be interpreted intuitively as a stochastic model for wind speed.  We determine the critical battery capacity threshold to be $3$ by using the techniques developed in \cite{Watkins2017b}.
	
	The evolution of $\{x_{u_g}(t)\}$ with initial condition $x_{u_g}(0) = 10,$ and $\Bcap$ varying from $\Bcrit -1$ to $\Bcrit$ is given in Figure \ref{fig:state_evo}.  By inspection, one can see that for $\Bcap = 2,$ $\{x_{u_g}(t)\}$ is unstable, as the sample expectation of the norm grows without bound, whereas for $\Bcap = 3$ the system is stable, with the expectation remaining below the bound $15.$  Unlike the case of a periodic source, one can note that the statistics appear to converge to a limiting distribution after an initial period of transience.  This is due to the ergodicity of the stochastic source model, which induces a stationary distribution in the energy arrival process, and hence in the state space process.
	\begin{figure}
		\begin{subfigure}[b]{\linewidth} 
			\centering\large 
			\includegraphics[width=\textwidth]{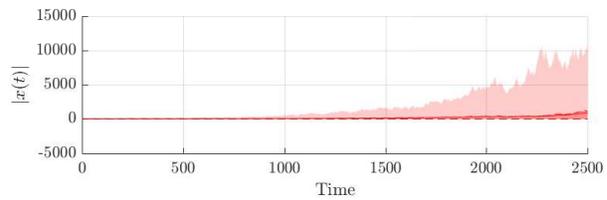}
			\caption{Battery Capacity $\Bcap = \Bcrit-1 = 2.$}
			\label{fig:1a} 
		\end{subfigure} 
		\begin{subfigure}[b]{\linewidth} 
			\centering\large 
			\includegraphics[width=\textwidth]{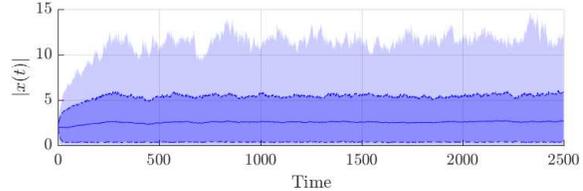}
			\caption{Battery Capacity $\Bcap = \Bcrit = 3.$}
			\label{fig:1b} 
		\end{subfigure}
		\caption{\small  A plot of a $10\text{k}$ sample Monte Carlo simulation of the example EHCS, of the system detailed in Section \ref{subsec:ergodic}.  The mean trajectory for each plot is given by a solid line, the $80\%$ confidence interval in dark shade, with the $98\%$ confidence interval given in lighter shade.  Note that unlike in the cases of periodic sources, there is no apparent periodicity in the plant's evolution in this case.} \label{fig:ergodic}
	\end{figure}
	
	\subsection{Periodic Stochastic Energy Harvesting Sources}
	\label{subsec:periodic}
	In this subsection, we detail how our proposal of using processes of the type detailed in Section \ref{subsec:harvesting} for modeling stochastic energy harvesting sources can be applied to sources with macroscopic, stochastic periodic fluctuations.  This is the most abstract level of generality encapsulated by models with transition matrices structured as \eqref{eq:block_struc}.  Moreover, these are features common to applications which are beholden to daily use or availability cycles.  As a concrete example, we may consider the construction of a model of solar light intensity, wherein between sunset and sunrise there is insufficient light available to harvest significant energy.
	
	In this setting, we may assume that the intensity incident to the energy harvesting device's solar cell decomposes into two effects: the intensity which would be experienced by the solar cell on an ideal, cloudless day, and the dampening effect of clouds.  In light of this, we define $d(t)$ to be the solar intensity experienced by the sensor at time $t,$ on an ideal, cloudless day.  To model the effect of cloud coverage, we assume that the dampening effect of clouds evolves as an ergodic Markov chain $C$ with transition matrix $C,$ and decreases the intensity of the incident sunlight additively with respect to the ideal value $d(t),$ where in the case that at a particular time the cloud loss is more than $d(t),$ no energy is received. 
	
	By structuring the latent space transition matrix with the block structure given by \eqref{eq:block_struc} with $Q_j = C$ for all $j$ in $[\Pd],$ we see that we have a periodic stochastic process with periodicity $\Pd,$ and $n$ possible states at each time.  Note that, as before, the block structure given by \eqref{eq:block_struc} allows us to implicitly keep track of time, by way of noting in which latent state the process currently resides.  As such, we defined $t(\ell) \in [\Pd]$ to be the time with respect to the period of the process associated to the latent state $\ell.$  Letting $c(\ell)$ represent fraction of maximum cloud coverage dampening associated to the latent state $\ell,$ we have that the total amount of energy harvested by the sensor at a latent state $\ell$ is given by 
	\begin{equation}
	h(\ell) = \lproj \Intensity d(\tau(\ell)) - \Dampening c(\ell)\rproj_0^{\infty},
	\end{equation}
	where we implicitly define $\tau$ as a function which maps the latent state $\ell$ to the element of the period associated to $\ell,$ $\Intensity$ to be the maximum intensity of sunlight on a cloudless day, and $\Dampening$ to be the maximum dampening effect placed on the solar intensity due to clouds.
	
	As a particular example, we may take $\Intensity = 5,$ $\Dampening = 4,$ $\Pd = 48,$ $d(\tau)$ as
	\begin{equation}
	d(\tau) = \sin \left(\frac{2 \pi \tau}{48} \right),
	\end{equation}
	$L(t)$ as the ergodic Markov chain taking values on $\{1,2,3,4\},$ with transition matrix
	\begin{equation}
	\Ltrans =
	\left[
	\begin{matrix}
	0.70 & 0.15 & 0.00 & 0.00\\
	0.30 & 0.70 & 0.15 & 0.00\\
	0.00 & 0.15 & 0.70 & 0.30\\
	0.00 & 0.00 & 0.15 & 0.70
	\end{matrix}
	\right],
	\end{equation}
	loss function $c(\ell) = \frac{(\ell - 1)}{|\mathcal{L}| - 1},$ and transmission energy $\energythresh = 1.$  Note that the choice of $d$ as a trigonometric function of time is supported well by literature \cite{Iqbal1983}, however the function used usually explicitly depends on the coordinates of the device on the Earth, and its angle with respect to the surface, as well as the time of year.  We have chosen a sinusoid here for simplicity; other models can be incorporated just as easily.
	
	The simulated behavior of this model is given in Figure \ref{fig:solar_source}, where we have $A_o = 1.017,$ $A_c = 0.950,$ $\energythresh = 2,$ $\lambda = 0.98,$ and $\{\disturbance(t)\}$ as a sequence of i.i.d uniform random variables on $[-0.5,0.5].$  We see the macroscopic periodic effects we would expect to see of a solar charging process.  Namely, over each period of $24$ hours, there are approximately $12$ hours of sunlight of varying intensity, and $12$ hours of darkness, in which no energy is received by the sensor.  We plot the results of a simulated EHCS under the greedy power allocation policy in Figure \ref{fig:solar_sim}, we see the effects of this periodicity in a simulation of the system under the greedy transmission policy, in which we see performance degrade during the periods in which the system receives no energy, and performance improve when energy becomes available again.  However, as predicted, the state process remains bounded for all times when a sufficiently large battery is used, and becomes unbounded otherwise, where the required size of the battery may be calculated by the techniques in \cite{Watkins2017b}.  These observations support the theory presented earlier in the paper, and point to an area of future work, wherein stabilizing policies which optimize performance are investigated.
	
	\begin{figure}
		\begin{subfigure}[b]{\linewidth} 
			\centering\large 
			\includegraphics[width=\textwidth]{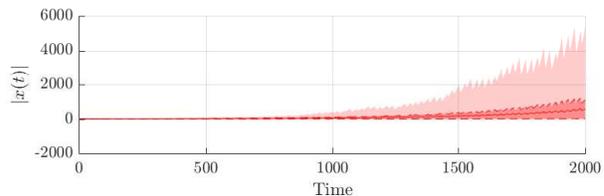}
			\caption{Battery Capacity $\Bcap = 0.$}
			\label{fig:1a} 
		\end{subfigure} 
		\begin{subfigure}[b]{\linewidth} 
			\centering\large 
			\includegraphics[width=\textwidth]{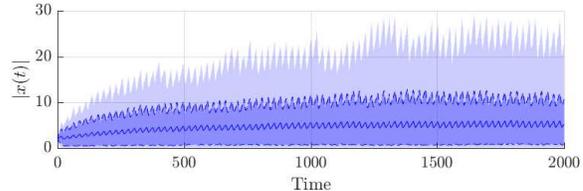}
			\caption{Battery Capacity $\Bcap = 1.$}
			\label{fig:1b} 
		\end{subfigure}
		\caption{A plot of $10$k sample Monte Carlo simulation of an example EHCS with a periodic stochastic source, as detailed in Section \ref{subsec:periodic}.  Note that the periodicity present in the source process is inherited by the plant state process, by way of passing through the greedy transmission policy.  Note also that for this system, the critical battery capacity was found to be $1,$ and which is confirmed by these simulations.} \label{fig:solar_sim}
	\end{figure}
	
	\section{Conclusions and Future Work}
	\label{sec:conclusion}
	In this paper, we established a computationally efficient means of certifying the stability of the evolution of a plant supplied with feedback signals by an energy harvesting sensor over a wireless communication channel under a fixed transmission policy.  We have shown that the developed test applies to any memoryless transmission policy.  As we have also proven that memoryless policies are all which are needed in order to stabilize scalar plants, and can be used to encode more complicated predictive policies capable of stabilizing more complicated plants, we believe it to be broad enough in this regard to be useful in practice.  
	
	Moreover, we have shown that the developed test applies to any system with an energy harvesting process which can be modeled as a function of a finite-state Markov chain, and that such processes can be used as models for several interesting sources including deterministic recharging, wind harvesting, and solar harvesting.  As such, the certification test we have developed is quite general in this regard as well, and we believe will be of use in future applications.  Future work can come in many directions, including considering a situation in which multiple sensors communicate to multiple plants, and generalizing the control model at the plant beyond simple linear feedback.
	\section*{Acknowledgments}
	{\small This is supported by the TerraSwarm Research Center, one of six centers supported by the STARnet phase of the Focus Center Research Program (FCRP), a Semiconductor Research Corporation program sponsored by MARCO and DARPA.}
	
	\bibliography{library}
	\bibliographystyle{ieeetr} 
	\appendix
	
	\subsection{Proof of Proposition \ref{prop:equiv}}
	\label{app:prop:equiv}
	By definition, exponential mean-square stability of $\{z_u(t)\}$ implies that for every initial condition $z_u(0),$ there exists some positive constant $\alpha_0 >0,$ and some constant in the open unit interval $\xi \in (0,1)$ such that $\E[z_u(t)^T z_u(t)] \leq \E[z(0)^T z(0)] \alpha \xi^t$ holds for all times $t.$  By expanding the dynamics of $\{z_u(t)\}$ appropriately, we see that
	\begin{equation}
	\begin{aligned}
	\E[z_u(t)^T z_u(t)] &\triangleq \E[z_u(0)^T (\Pi_{j = 0}^{t - 1} A_j)^T \left(\Pi_{j = 0}^{t - 1} A_j\right) z_u(0)] \\
	&\leq  \E[z(0)^T z(0)] \alpha \xi^{t}
	\end{aligned}
	\end{equation}
	holds as well.  By definition, we have that the expectation of $\{x_u(t)\}$ evolves as
	\begin{equation*}
	\begin{aligned}
	&\E[x_u^T(t) x_u(t)] \triangleq \E[x_u(0)^T (\Pi_{j = 0}^{t - 1} A_j)^T \left(\Pi_{j = 0}^{t - 1} A_j\right) x_u(0)] \\
	& +\sum_{j = 0}^{t-1} \E[\omega_j^T (\Pi_{k = j-1}^{t - 1} A_k)^T(\Pi_{k = j-1}^{t - 1} A_k) \omega_j],
	\end{aligned}
	\end{equation*}
	which after applying the linearity and cycle-invariance properties of the trace operator, along with the fact that the disturbance process is i.i.d. with second moment matrix $W$ becomes
	\begin{equation*}
	\begin{aligned}
	\E[x_u^T(t) x_u(t)] &\triangleq \E[x_u(0)^T (\Pi_{j = 0}^{t - 1} A_j)^T \left(\Pi_{j = 0}^{t - 1} A_j\right) x_u(0)] \\
	& + \sum_{j = 0}^{t-1} \Tr \left(\E[(\Pi_{k = j+1}^{t - 1} A_k)^T(\Pi_{k = j+1}^{t - 1} A_k)] W \right).
	\end{aligned}
	\end{equation*}
	By noting that for each time $j,$ the term $\Tr \left(\E[(\Pi_{k = j+1}^{t - 1} A_k)^T(\Pi_{k = j+1}^{t - 1} A_k)] W \right)$ corresponds exactly to an undisturbed EHCS with initial condition $\omega_j,$ we have that exponential mean-square stability of ${z_u(t)}$ implies that 
	\begin{equation}
	\label{ineq:x_dist1}
	\begin{aligned}
	\E[x_u^T(t) x_u(t)] &\leq \E[x_u^T(0) (\Pi_{j = 0}^{t - 1} A_j)^T \left(\Pi_{j = 0}^{t - 1} A_j\right) x_u(0)] \\
	&+ \sum_{j = 0}^{t-1} \alpha \Tr\left(W\right) \xi^{t-j-1}
	\end{aligned}
	\end{equation}
	holds.  Focusing now only on the rightmost sum after having factored out $\alpha \Tr\left( W\right),$ we see that
	\begin{equation}
	\label{ineq:x_dist2}
	\begin{aligned}
	\lim_{t \goesto \infty} \sum_{j = 0}^{t-1} \xi^{t-j-1} = \lim_{t \goesto \infty} \sum_{r = 0}^{t-1} \xi^{r} = \frac{1}{1 - \xi}
	\end{aligned}
	\end{equation}
	holds, where the equality comes from the well-known summation formula for geometric sums.  Noting that by definition, $z_u(0) = x_u(0),$ we have that the exponential mean-square stability of $\{z_u(t)\}$ implies that
	\begin{equation}
	\label{ineq:x_homo}
	\E[x_u^T(0) (\Pi_{j = 0}^{t - 1} A_j)^T \left(\Pi_{j = 0}^{t - 1} A_j\right) x_u(0)] \leq \E[x_u^T(0) x_u(0)] \alpha \xi^t 
	\end{equation}
	holds.  Putting \eqref{ineq:x_homo} together with \eqref{ineq:x_dist1} and \eqref{ineq:x_dist2} shows that
	\begin{equation}
	\begin{aligned}
	\E[x_u^T(t) x_u(t)] &\leq \E[x_u^T(0) x_u(0)] \alpha \xi^t + \frac{\alpha \Tr \left(W \right)}{1 - \xi}
	\end{aligned}
	\end{equation}
	holds, which confirms that $\{x_u(t)\}$ is exponentially mean-square ultimately bounded with $M = \frac{\alpha}{1 - \xi},$ as claimed.
	
	We now show that exponential mean-square ultimate boundedness of $\{x_u(t)\}$ implies exponential mean-square stability of $z_u(t).$  By definition, exponential mean-square ultimate boundedness of $\{x_u(t)\}$ implies that $\E[x_u^T(t)x_u(t)]$ is bounded asymptotically.  Hence, we have that the sum defining the contribution of disturbances to the state vector process is bounded, i.e.
	\begin{equation}
	\label{ineq:bound1}
	\lim_{t \goesto \infty} \sum_{j = 0}^{t-1} \Tr \left(\E[(\Pi_{k = j+1}^{t - 1} A_k)^T(\Pi_{k = j+1}^{t - 1} A_k)] W \right) < \infty,
	\end{equation}
	holds.  By defining the matrix $\Trans{k}{r}$ as
	\begin{equation}
	\Trans{t-j}{r} \triangleq \E[(\Pi_{k = 0}^{t - j} \Amode{k})^T(\Pi_{k =0}^{t - j} \Amode{k})|\mode{0} = r],
	\end{equation}
	where $r$ may take any value on the set of positive recurrent modes $\Recurrent,$ we have that \eqref{ineq:bound1} implies the weaker bound
	\begin{equation}
	\lim_{t \goesto \infty} \sum_{r \in \Recurrent} \sum_{j = 0}^{t-1} \Tr \left(\Trans{t-j}{r} W \right) \Pr(\mode{j+1} = r) < \infty.
	\end{equation}
	Since the mode transition process will enter the positive recurrent set almost surely in finite time, and thereafter each mode in the positive recurrent set of modes will have a strictly positive probability, we may assume that 
	\begin{equation}
	\label{ineq:bound2}
	\lim_{\bar{k} \goesto \infty} \sum_{k = 0}^{\bar{k}} \Tr \left(\Trans{k}{r} W \right) \pmin < \infty 
	\end{equation}
	holds for some $\pmin > 0,$ where if process does not start in the positive recurrent set, we must shift time by a finite amount for the above to hold.  The bound \eqref{ineq:bound2} implies that
	\begin{equation}
	\lim_{k \goesto \infty}  \Tr \left(\Trans{k}{r} W \right) = 0 
	\end{equation}
	holds.  We now show that this implies that $\Trans{k}{r}$ converges to $0$ in the limit of large $k.$  In particular, since $\Trans{k}{r}$ is symmetric positive semidefinite and $W$ is symmetric positive definite, we have that $\Trans{k}{r}$ posses a symmetric positive semidefinite square root, and $W$ possesses a symmetric positive definite square root.  Decomposing $\Trans{k}{r}$ and $W$ into their square roots, cycling the arguments of the trace operator, and rewriting demonstrates the identity 
	\begin{equation}
	\begin{aligned}
	\Tr \left(\Trans{k}{r} W \right) &= \Tr \left(\sqrt{W} \sqrt{\Trans{k}{r}} \sqrt{\Trans{k}{r}} \sqrt{W} \right)\\
	&= \sum_{i = 1}^{n} \|\sqrt{\Trans{k}{r}} [\sqrt{W}]_{i} \|_2^2.
	\end{aligned}
	\end{equation}
	Since $\sqrt{W}$ is positive definite, it follows that it is full rank.  Hence, at least one column $[\sqrt{W}]_{i^\star}$ of $\sqrt{W}$ is not in the nullspace of $\sqrt{\Trans{k}{r}}$ and thus $\Tr \left(\Trans{k}{r} W \right) > 0,$ unless $\Trans{k}{r} = 0.$  It then follows that $\lim_{k \goesto \infty}  \Trans{k}{r} = 0$ holds as claimed.  Moreover, since $$\E[z_u^T(t)z_u(t)| \mode{0} = r] = z_u^T(0)\Trans{t}{r} z_u(0),$$ for any choice $r \in \Recurrent,$ and the finite number of modes implies that the chain enters the positive recurrent set $\Recurrent$ in finite time, it follows that $\{z_u(t)\}$ is asymptotically mean-square stable.  Since $\{z_u(t)\}$ is an unperturbed Markov jump linear system, it holds that asymptotic mean-square stability and exponential mean-square stability are equivalent (see, e.g., \cite[Theorem 3.9]{Costa2005}).  Hence, the exponential mean-square ultimate boundedness of $\{x_u(t)\}$ implies the exponential mean-square stability of $\{z_u(t)\},$ as claimed. \oprocend 
	
	\subsection{Proof of Theorem \ref{thm:linear_prog}}
	\label{app:thm:linear_prog}
	We specialize \cite[Proposition 3.42]{Costa2005} to claim that a MJLS with mode process $\{S(t)\}$ on a state space $\mathcal{S}$ is mean-square stable if and only if there exist positive definite matrices $R_s$ such that the inequality
	\begin{equation}
	\label{ineq:mjls_stab}
	A_s^T \sum_{s^\prime \in \mathcal{S}} \Pr(S(t+1) = s^\prime | S(t) = s) R_{s^\prime} A_s \curlyleq R_s
	\end{equation}
	holds for all $s \in \mathcal{S}.$  Allowing $R$ to denote a set of matrices containing each $R_s,$ we note that since \eqref{ineq:mjls_stab} is affine in $R,$ we may scale any $R$ which satisfies \eqref{ineq:mjls_stab} by an arbitrary positive constant,  
	and so \eqref{ineq:mjls_stab} is equivalent to
	\begin{equation}
	\label{ineq:mjls_stab2}
	\begin{aligned}
	A_s^T \sum_{s^\prime \in \mathcal{S}} \Pr(S(t+1) = s^\prime | S(t) = s) R_{s^\prime} A_s - R_s \curlyleq v\\
	R_s \curlygeq I,\quad
	v < 0.
	\end{aligned}
	\end{equation}
	Since the system of inequalities \eqref{ineq:mjls_stab2} is affine in $R,$ we can make the left hand side of the first inequality of \eqref{ineq:mjls_stab2} an arbitrarily large negative number if \eqref{ineq:mjls_stab2} is satisfied for any $R.$  Hence, if the system's evolution is stable, then the semidefinite program \eqref{prog:sdp} is unbounded below.  Conversely, if \eqref{prog:sdp} is unbounded below, it follows that an $R$ exists such that \eqref{ineq:mjls_stab2} is satisfied, and the system's evolution is stable. We complete the proof by noting the preceding arguments immediately imply that the lower bound $v \geq -D$ of \eqref{prog:sdp} must saturate if and only if the MJLS is stable.\oprocend 
	
	\subsection{Proof of Lemma \ref{lem:pathwise_dominance}}
	\label{app:lem:pathwise_dominance}
	Consider an arbitrary element $\omega \in \Omega,$ and note that for the remainder of this proof, all stochastic process variables are evaluated with respect to $\omega,$ though we do not explicitly notate the dependence.  Letting $N_{u}^{\{S\}}(t)$ be a random variable denoting the number of successful loop closures which have occurred through time $t$ under transmission policy $u,$ and $N_{u}^{\{A\}}(t)$ be the number of attempted loop closures through time $t$ under transmission policy $u,$ we see that we have
	$N_{u}^{\{S\}}(t) = \sum_{\ell = 0}^{N_u^{\{A\}}(t)} g_\omega(\ell),$
	i.e., that the number of successful loop closures attained by a policy $u$ through time $t$ is simply the sum of the first $N_u^{\{A\}}(t)$ packet reception indicator random variables on the sample path.  
	
	Importantly, this implies that along any particular sample path $\omega,$ it suffices to check that $N_{u}^{\{A\}}(t) \leq N_{u^\prime}^{\{A\}}(t)$ holds for all $t$ to verify that $N_{u}^{\{S\}}(t) \leq N_{u^\prime}^{\{S\}}(t)$ holds for all $t.$  In light of this, we argue that the inequality
	\begin{equation}
	\label{ineq:induct1}
	N_{u}^{\{A\}}(t) \leq N_{u_g}^{\{A\}}(t),
	\end{equation}
	holds for all $t \geq 0.$ Since arguing the validity of \eqref{ineq:induct1} directly is difficult, we also prove the validity of an energy storage inequality to help.  To be more precise, we show that
	\begin{equation}
	\label{ineq:induct2}
	B_{u}(t+1) \leq B_{u_g}(t+1) + \energythresh(N_{u_g}^{\{A\}}(t) - N_{u}^{\{A\}}(t)),
	\end{equation}
	holds for all $t,$ and so that if $N_{u_g}^{\{A\}}(t) = N_{u}^{\{A\}}(t),$ then it must be $B_u(t+1) \leq B_{u_g}(t+1).$
	
	By definition, we have that $B_{u}(0) = B_{u_g}(0),$ and so it follows that $N_{u}^{\{A\}}(0) \leq N_{u_g}^{\{A\}}(0),$ as by the equality of the battery levels, any transmission policy $u$ may supply a feedback signal to the communication channel at time $0$ only if the greedy transmission policy $u_g$ supplies a feedback signal to the communication channel at time $0.$  Moreover, it holds that $B_{u}(1) \leq B_{u_g}(1) + \energythresh(N_{u_g}^{\{A\}}(0) - N_u^{\{A\}}(0)),$
	as if both policies supply  a feedback signal to the transmitter, then it must hold that $B_u(1) \leq B_{u_g}(1),$ and if the transmission policy $u$ does not provide a feedback signal to the transmitter while $u_g$ does, we have that $B_{u}(1) \leq B_{u_g}(1) + \energythresh,$ which follows from noting that the system uses exactly $\energythresh$ units of energy under policy $u_g.$
	
	For purposes of compacting notation, let $\tau^{-} = \tau - 1 $ and $\tau^+ = \tau+1.$  Take the preceding argument as a base for induction, suppose that \eqref{ineq:induct1} and \eqref{ineq:induct2} hold for all $t \leq \tau-1.$  We can infer immediately that $N_{u}^{\{A\}}(\tau^{-}) \leq N_{u_g}^{\{A\}}(\tau^{-})$ and $B_{u}(\tau) \leq B_{u_g}(\tau) + \energythresh(N_{u_g}^{\{A\}}(\tau^{-}) - N_{u}^{\{A\}}(\tau^{-}))$ together imply that $N_{u}^{\{A\}}(\tau) \leq N_{u_g}^{\{A\}}(\tau)$ holds, since the transmission policy $u$ may only supply a feedback signal to the sensor at time $\tau$ if the greedy transmission policy $u_g$ does so as well.  It remains to prove that $B_{u}(\tau^{+}) \leq B_{u_g}(\tau^{+}) + \energythresh(N_{u_g}^{\{A\}}(\tau) - N_u^{\{A\}}(\tau))$
	holds for $t = \tau.$
	
	Define $\Delta_{u}^{\{A\}}(\tau) \triangleq N_u^{\{A\}}(\tau^+) - N_u^{\{A\}}(\tau).$  Observe that 
	\begin{equation}
	\label{ineq:induct4_1}
	\energy_{u}(\tau^+) \geq \energythresh \Delta_{u}^{\{A\}} (\tau),
	\end{equation}
	as if $\Delta_{u}^{\{A\}} (\tau)= 1,$ then $\energy_{u}(\tau^+) \geq \energythresh,$ and otherwise $\energy_{u}(\tau^+) \geq 0.$  By multiplying both sides of \eqref{ineq:induct4_1} by $-1$ and adding $\energythresh \Delta_{u_g}^{\{A\}}(\tau),$ we have
	\begin{equation*}
	\label{ineq:induct3}
	\energythresh\Delta_{u_g}^{\{A\}}(\tau) - \energy_{u}(\tau) \leq \energythresh[\Delta_{u_g}^{\{A\}}(\tau) - \Delta_{u}^{\{A\}} (\tau)].
	\end{equation*}
	By noting that $\energy_{u_g}(\tau)$ is defined as $\energythresh \Delta_{u_g}^{\{A\}}(\tau),$ we have
	\begin{equation*}
	\energy_{u_g}(\tau) - \energy_{u}(\tau) \leq \energythresh[\Delta_{u_g}^{\{A\}}(\tau)- \Delta_{u}^{\{A\}}(\tau)],
	\end{equation*}
	and finally
	\begin{equation}
	\label{ineq:induct4}
	\begin{aligned}
	& - \energy_{u}(\tau) + \energythresh (N_{u_g}^{\{A\}}(\tau^{-}) - N_{u}^{\{A\}}(\tau^{-}) )\leq \\
	& - \energy_{u_g}(\tau) + \energythresh(N_{u_g}^{\{A\}}(\tau) - N_{u}^{\{A\}}(\tau)),
	\end{aligned}
	\end{equation}
	which we use to finish the induction.  In particular, we have
	\begin{equation*}
	\begin{aligned}
	&B_{u}(\tau^+) \leq \\
	&\lproj B_{u_g}(\tau) + H(\tau) - \energy_u(\tau) + \energythresh(N_{u_g}^{\{A\}}(\tau^{-}) - N_{u}^{\{A\}}(\tau^{-}))\rproj_{0}^{\Bcap}\\
	& \leq\\
	&\lproj B_{u_g}(\tau) + H(\tau) - \energy_{u_g}(\tau) + \energythresh(N_{u_g}^{\{A\}}(\tau) - N_{u}^{\{A\}}(\tau))\rproj_{0}^{\Bcap}\\
	&\leq B_{u_g}(\tau^+) + \energythresh(N_{u_g}^{\{A\}}(\tau) - N_{u}^{\{A\}}(\tau)),
	\end{aligned}
	\end{equation*}
	where in the first inequality we have used the non-negativity of $\energythresh(N_{u_g}^{\{A\}}(\tau^{-}) - N_{u}^{\{A\}}(\tau^{-})),$ and in the second we have used \eqref{ineq:induct4}.  As this demonstrates that \eqref{ineq:induct2} holds with $t = \tau,$ the induction is complete.  Since $\omega$ was chosen arbitrarily, it holds for every $\omega \in \Omega,$ which completes the proof.
	\oprocend
\end{document}